\documentclass[12pt,reqno]{article}

\usepackage[usenames]{color}
\usepackage{amssymb}
\usepackage{amsmath}
\usepackage{amsthm}
\usepackage{amsfonts}
\usepackage{amscd}
\usepackage{graphicx}
\usepackage{booktabs}
\usepackage[colorlinks=true,
linkcolor=webgreen,
filecolor=webbrown,
citecolor=webgreen]{hyperref}

\definecolor{webgreen}{rgb}{0,.5,0}
\definecolor{webbrown}{rgb}{.6,0,0}

\usepackage{color}
\usepackage{fullpage}
\usepackage{float}

\usepackage{graphics}
\usepackage{latexsym}

\usepackage{xcolor}

\newcommand{\braces}{\genfrac{\lbrace}{\rbrace}{0pt}{}}

\allowdisplaybreaks

\begin{document}

\theoremstyle{plain}
\newtheorem{theorem}{Theorem}
\newtheorem{proposition}[theorem]{Proposition}
\newtheorem{corollary}[theorem]{Corollary}
\newtheorem{lemma}[theorem]{Lemma}
\theoremstyle{definition}
\newtheorem{example}[theorem]{Example}
\newtheorem*{remark}{Remark}

\numberwithin{equation}{section}
\numberwithin{theorem}{section}

\begin{center}
\vskip 1cm
{\Large\bf Double sums associated with binomial transforms}

\vskip 1cm

{\large
Kunle Adegoke \\
Department of Physics and Engineering Physics \\ Obafemi Awolowo University \\ 220005 Ile-Ife, Nigeria \\
\href{mailto:adegoke00@gmail.com}{\tt adegoke00@gmail.com} \\
ORCID: 0000-0002-3933-0459

\vskip .25 in

Robert Frontczak \\
Independent Researcher \\ 72764 Reutlingen,  Germany \\
\href{mailto:robert.frontczak@web.de}{\tt robert.frontczak@web.de}\\
ORCID: 0000-0002-9373-5297

\vskip .25 in

Karol Gryszka \\
Institute of Mathematics \\ University of the National Education Commission, Krakow \\
Podchor\c{a}\.{z}ych 2, 30-084 Krak{\'o}w, Poland\\
\href{mailto:email}{\tt karol.gryszka@uken.krakow.pl} \\
ORCID: 0000-0002-3258-3330

\vskip .25 in
}

\end{center}

\vskip .2 in

\begin{abstract}
In this paper, we continue our investigation of double sums where the inner sum is binomial but incomplete. We prove many new results for these types of double sums associated with binomial transform pairs. As applications we deduce new identities for double sums involving special numbers like Bernoulli numbers, Fibonacci numbers, harmonic numbers, Catalan numbers and Stirling numbers of the second kind. 
We also consider families of polynomials like Fibonacci polynomials, Chebyshev polynomials, Bernoulli polynomials, and others. 
Finally, we state new double sums involving hyperbolic functions.
\end{abstract}



\section{Introduction}

For $n=0,1,2,\dots$, let $(s_n)_{n\geq 0}$ and $(\sigma_n)_{n\geq 0}$ be two sequences of complex numbers. 
If these sequences are connected by the relations
\begin{equation*}
\sigma_n = \sum_{k = 0}^n \binom{n}{k} (-1)^k s_k, \qquad\Leftrightarrow\qquad s_n = \sum_{k = 0}^n \binom{n}{k} (-1)^k \sigma_k,
\end{equation*}
then $\{s_n,\sigma_n \}$ are known as and will be called a binomial transform pair. In addition, if the sequence $s_n$ satisfies the relation
\begin{equation}\label{self}
\sum_{k = 0}^n \binom{n}{k} (-1)^k s_k = s_n \qquad (n=0,1,2,\dots),
\end{equation}
then we say that $s_n$ is a self-inverse sequence. If 
\begin{equation}\label{anti_self}
\sum_{k = 0}^n \binom{n}{k} (-1)^k s_k = - s_n, \qquad (n=0,1,2,\dots),
\end{equation}
then $s_n$ will be called an anti-self-inverse sequence. \\

Although the topic is surely much older, the study of binomial transforms has attracted the interest of researchers in the recent past. 
We mention the papers by Gould \cite{Gould2}, Haukkanen \cite{Haukkanen} and Prodinger \cite{Prodinger}, who started to produce 
new results back in the 1990ies. Wang \cite{Wang}, Chen \cite{chen07}, the Sun brothers \cite{Sun1,Sun2,Sun3}, Mu \cite{Mu}, 
Boyadzhiev \cite{Boya1,Boya2,Boya3} and Adegoke \cite{Adegoke0} provided further results in this century. Boyadzhiev devoted the topic 
a whole book \cite{Boya4}. 

As many derivatives of important number sequences like Bernoulli numbers $B_n$ or Fibonacci numbers $F_n$ can be treated as binomial transform pairs, binomial transforms are prominently placed in papers about number theory, combinatorics and probability. Two particular results about the binomial transform pair $\{s_n,\sigma_n \}$ which are linked to our study are the following identities: 
For all non-negative integers $m$ and $n$ we have
\begin{equation}\label{s4v4rqn}
\sum_{k=0}^m \binom{m}{k} (-1)^k s_{n+k} = \sum_{j=0}^n \binom{n}{j} (-1)^j \sigma_{m+j}.
\end{equation}
For all complex numbers $p$ we also have the double sum identity
\begin{equation*}
p^{n+m} \sum_{j=0}^n \binom{n}{j} p^{-j} \sum_{k=0}^m \sigma_{j+k} p^{-k} 
= (p+1)^{n+m} \sum_{j=0}^n \binom{n}{j} (p+1)^{-j} \sum_{k=0}^m (-1)^{j+k} s_{j+k} (p+1)^{-k},
\end{equation*}
with the particular case being
\begin{equation*}
\sum_{j=0}^n \binom{n}{j} \sigma_{j} p^{n-j} = \sum_{j=0}^n \binom{n}{j} (-1)^j s_{j} (p+1)^{n-j}.
\end{equation*}
Both results can be proved via generating functions. Identity~\eqref{s4v4rqn} was first obtained by Chen~\cite[Proof of Theorem 3.1, generalized in Theorem 3.2]{chen07}; see also Gould and Quaintance~\cite{gould14}. In this paper, we deal with similar relations involving double sums of the form
\begin{equation}\label{double_sum}
\sum_{k=0}^n x^k \left (\sum_{j=0}^k \binom{n}{j} y^j\right ) \qquad k\leq n,
\end{equation}
where $x$ and $y$ are complex numbers. As is seen, the inner sum in \eqref{double_sum} is binomial but incomplete. 
This inner sum has no simple closed form. However, it can be expressed in terms of the Gaussian hypergeometric function ${}_2F_1$ 
defined for complex numbers $a,b,c$ and $z$ by
\begin{equation*}
{}_2F_1(a,b;c;z) = \sum_{n=0}^\infty \frac{(a)_n (b)_n}{(c)_n} \frac{z^n}{n!},
\end{equation*}
where $(a)_n$ is the Pochhammer symbol given by $(a)_n = \Gamma(a+n)/\Gamma(a)$, and $\Gamma(a)$ is the Gamma function. 
Following \cite{Stenlund} we have
\begin{equation*}
\sum_{j=0}^{k} \binom{n}{j} y^j = \frac{y^{k+1}}{y+1} \binom{n}{k} \,{}_2F_1\left(1,n+1;n+1-k;\tfrac{1}{y+1}\right).
\end{equation*}

In this paper, we continue our investigations started in \cite{Adegoke3}. 
We prove a range of new identities for double sums associated with binomial transform pairs.
As applications we deduce new identities for double sums involving special numbers like Bernoulli numbers $B_n$, Fibonacci numbers $F_n$, 
harmonic numbers $H_n$, Catalan numbers $C_n$ and Stirling numbers of the second kind. To give the reader a sense of what will follow, we highlight four of our findings as examples:
\begin{equation*}
\sum_{k = 0}^n \sum_{j = 0}^k \binom{n}{j} (-1)^k B_{k-j} = B_n + B_{n-1}, \qquad (n\geq 1),
\end{equation*}
\begin{equation*}
\sum_{k = 0}^n \sum_{j = 0}^k \binom{n}{j} (- 1)^j F_{k-j} = (-1)^{n-1} F_{n-2}, \qquad (n\geq 1),
\end{equation*}
\begin{equation*}
\sum_{k = 0}^n \sum_{j = 0}^k \binom{n}{j} (-1)^j H_{k + r - j} = (-1)^n
 \begin{cases}
 - H_{r+1}, & \text{if $n=1$;}  \\ 
 \frac{1}{n-1} \binom{n + r - 1}{r}^{-1} - \frac{1}{n} \binom{n + r}{r}^{-1}, &\text{if $n\geq 2$,} 
 \end{cases} 
\end{equation*}
and
\begin{equation*}
\sum_{k = 0}^n \sum_{j = 0}^k \binom{n}{j} (-1)^j 2^{2(k - j)} C_{n - k + j + 1} = 4C_n. 
\end{equation*}
We also study polynomial sequences like Bernoulli polynomials, Fibonacci polynomials or Chebyshev polynomials.
Finally, we state new double sums involving hyperbolic functions.

\section{Preliminaries}

This section contains the definitions and basic relations for the quantities used in the main text. 

The Catalan numbers $C_n$ are closely related to central binomial coefficients. They are commonly presented as 
$$C_{n} = \frac{1}{n+1} \binom{2n}{n},
$$  
but can also be expressed by the recursion
\begin{equation*}
C_n = \frac{2(2n-1)}{n+1} C_{n-1}, \quad C_0 = 1.
\end{equation*}

The Bernoulli numbers, $B_n$, are defined by the generating function
\begin{equation*}
\frac{z}{e^z - 1} = \sum_{n = 0}^\infty B_n \frac{z^n}{n!}, \quad |z| < 2\pi,
\end{equation*}
and the Bernoulli polynomials $B_n(x)$ by the generating function
\begin{equation*}
\frac{z e^{xz}}{e^z - 1} = \sum_{n = 0}^\infty B_n(x) \frac{z^n}{n!}, \quad |z| < 2\pi.
\end{equation*}
Clearly, $B_n=B_n(0)$. An explicit formula for the Bernoulli polynomials is
\begin{equation}\label{nskd3u8}
B_n (x) = \sum_{k = 0}^n \binom{n}{k} B_k x^{n-k},
\end{equation}
while a recurrence formula for them is
\begin{equation*}
B_n (x + 1) = \sum_{k = 0}^n \binom{n}{k} B_k (x),
\end{equation*}
or more generally,
\begin{equation}\label{de8ucu6}
B_n(x + y) = \sum_{k = 0}^n \binom{n}{k} B_k (x) y^{n-k}.
\end{equation}
We also have the relation
\begin{equation}\label{Ber_bt}
\sum_{k = 0}^n \binom{n}{k} B_k = (- 1)^n B_n.
\end{equation}

The Fibonacci numbers $F_n$ and the Lucas numbers $L_n$ are defined, for \text{$n\in\mathbb Z$}, through the recurrence relations
\begin{align*}
F_n = F_{n-1}+F_{n-2},& \quad n\geq 2,\quad F_0=0,\, F_1=1,\\
L_n = L_{n-1}+L_{n-2},& \quad  n\geq 2, \quad L_0=2, \, L_1=1.
\end{align*}
The Binet formulas for these sequences are
\begin{equation*}
F_n = \frac{\alpha^n - \beta^n}{\alpha - \beta}, \qquad L_n = \alpha^n + \beta^n, \quad n\in\mathbb Z,
\end{equation*}
with $\alpha=\frac{1+\sqrt 5}2$ being the golden ratio and $\beta=-\frac{1}{\alpha}$. For negative subscripts we have 
$F_{-n} = (-1)^{n-1}F_n$ and $L_{-n} = (-1)^n L_n$. These famous sequences are indexed as sequences {A000045} and {A000032} 
in the On-Line Encyclopedia of Integer Sequences \cite{OEIS}. A huge amount of further information about them can be found 
in the books by Koshy \cite{Koshy} and Vajda \cite{Vajda}, for instance. Also recall that the Gibonacci sequences has the same recurrence relation as the Fibonacci sequence but starts with arbitrary initial values, i.e.,
\begin{equation*}
G_k = G_{k - 1} + G_{k - 2},\qquad (k \ge 2),
\end{equation*}
with $G_0$ and $G_1$ arbitrary numbers (usually integers) not both zero. When $G_0=0$ and $G_1=1$ then $G_n=F_n$, and when 
$G_0=2$ and $G_1=1$ then $G_n=L_n$, respectively. The sequence obeys the generalized Binet formula
\begin{equation*}
G_n = A\alpha^n + B\beta^n,
\end{equation*}
where $A=\tfrac{G_1-G_0\beta}{\alpha-\beta}$ and $B=\tfrac{G_0\alpha-G_1}{\alpha-\beta}$. 

For $m,s\in\mathbb C$, the harmonic numbers of order $m$, $H_s^{(m)}$, and odd harmonic numbers of order $m$, $O_s^{(m)}$, 
are defined by
\begin{equation*}
H_s^{(m)} = H_{s - 1}^{(m)} + \frac{1}{s^m} \qquad \text{and} \qquad O_s^{(m)} = O_{s - 1}^{(m)} + \frac{1}{(2s - 1)^m},
\end{equation*}
with $H_0^{(m)}=0$ and $O_0^{(m)}=0$. The recurrence relations imply that if $s=n$ is a non-negative integer, then
\begin{equation*}
H_n^{(m)} = \sum_{j = 1}^n \frac{1}{j^m} \qquad \text{and} \qquad O_n^{(m)} = \sum_{j = 1}^n \frac{1}{(2j - 1)^m}.
\end{equation*}
For $m=1$, $H_s=H_s^{(1)}$ and $O_s=O_s^{(1)}$ are the ordinary harmonic and odd harmonic numbers, respectively. 
Harmonic numbers are connected to the digamma function $\psi(s)=\Gamma'(s)/\Gamma(s)$ through the fundamental relation
\begin{equation}\label{Harm_psi}
H_s = \psi(s + 1) + \gamma,
\end{equation}
where $\gamma$ is the Euler-Mascheroni constant.

\section{New results associated with binomial transform pairs}

In this section, we collect our main results for double sums involving binomial transform pairs $\{s_n,\sigma_n \}$.
Applications for special sequences of numbers and polynomials will be given in the proceeding Sections 4 and 5. 

To prove our results we exploit the main statement of \cite{Adegoke3} which is Theorem 3.1 in \cite{Adegoke3}. 
It is reproduced here for the reader's convenience:
\begin{theorem}\label{main_thm1}
For all complex numbers $x$ and $y$ we have
\begin{equation}\label{main_id1}
\sum_{k=0}^n x^k \sum_{j=0}^k \binom{n}{j} y^j = \frac{1}{1-x}\left ( (1+xy)^n - x^{n+1} (1+y)^n \right ).
\end{equation}
\end{theorem}

\begin{corollary}
If $n$ is a positive integer and $y$ is a complex number, then
\begin{equation}\label{gr3jv5j}
\sum_{k = 0}^n \sum_{j = 0}^k \binom{n}{j} (- 1)^j (1-y)^j = n y^{n-1} + y^n.
\end{equation}
\end{corollary}
\begin{proof}
Multiply through~\eqref{main_id1} by $1-x$, differentiate with respect to $x$, set $x=1$ and make a replacement $y\mapsto 1-y$.
\end{proof}

\begin{lemma}\label{main_lem}
Let $\left\{s_n,\sigma_n \right\}$, $n=0,1,2,\ldots$, be a binomial transform pair. Let $\mathcal L_y$ be a linear operator 
defined by $\mathcal L_y(y^j)=s_j$ for every complex number $y$ and every non-negative integer~$j$. Then $\mathcal L_y((1-y)^j)=\sigma_j$.
\end{lemma}
\begin{proof}
We have
\begin{align*}
\mathcal L_y\left( (1 - y)^j \right) = \mathcal L_y\left( \sum_{k = 0}^j \binom{j}{k} (- 1)^k y^k \right) 
= \sum_{k = 0}^j \binom{j}{k} (- 1)^k \mathcal L_y (y^k) = \sum_{k = 0}^j \binom{j}{k} (- 1)^k s_k =\sigma_j.
\end{align*}
\end{proof}

\begin{theorem}\label{thm.qtnqdr9}
Let $\{s_n,\sigma_n \}$, $n=0,1,2,\dots$, be a binomial transform pair. If $n$ is a positive integer and $m$ is a non-negative integer, then
\begin{equation}
\sum_{k = 0}^n \sum_{j = 0}^k ( - 1)^j \binom{{n}}{j} s_{m + n - k + j} = \sum_{k = 0}^m (-1)^k\binom{{m}}{k} \sigma_{n + k - 1}. 
\end{equation}
In particular,
\begin{equation}\label{utju7bn}
\sum_{k = 0}^n \sum_{j = 0}^k (- 1)^j \binom{n}{j} s_{n + j - k} = \sigma_{n - 1}.
\end{equation}
\end{theorem}
\begin{proof}
In \eqref{main_id1} make the replacements $x\mapsto 1/y$ and $y\mapsto -y$. This gives after multiplying through by $y^n$
\begin{equation}\label{hctczju}
\sum_{k=0}^n \sum_{j=0}^k \binom{n}{j} (-1)^j y^{n-k+j} = (1-y)^{n-1}.
\end{equation}
Now, apply Lemma \ref{main_lem}. The final identity is obtained by noticing from~\eqref{s4v4rqn} that the sequences
\begin{equation*}
s_{m + k} \qquad \text{and} \qquad \sum_{j = 0}^m \binom{{m}}{j} (- 1)^j \sigma_{j + k}
\end{equation*}
are a binomial transform pair.
\end{proof}

\begin{corollary}\label{cor.gca1g47}
Let $\{s_n,\sigma_n \}$, $n=0,1,2,\ldots$, be a binomial transform pair. If $n$ is a positive integer, then
\begin{equation}\label{vsjufnd}
\sum_{k = 0}^n \sum_{j = 0}^k (- 1)^j \binom{n}{k} H_{n + j - k} s_{n + j - k} = \sigma_{n - 1} H_{n - 1} - \sum_{k = 1}^{n - 1} 
\frac{\sigma_{n - 1 - k}}{k}.
\end{equation}
\end{corollary}
\begin{proof}
Boyadzhiev~\cite{Boya3} has shown that if $\{s_n,\sigma_n \}$, $n=0,1,2,\ldots$, is a binomial transform pair, then
\begin{equation*}
\sum_{k = 0}^n (- 1)^k \binom nk H_k s_k = \sigma_n H_n - \sum_{k = 1}^n \frac{\sigma_{n - k}}{k}.
\end{equation*}
Use this result in~\eqref{utju7bn}.
\end{proof}

\begin{theorem}\label{thm.ef1drze}
Let $(s_n)$ and $(t_n)$, $n=0,1,2,\ldots$, be two sequences of complex numbers. If $\{s_n,\sigma_n \}$ and $\{t_n,\tau_n \}$ are binomial pairs, then
\begin{equation}\label{eq:2ndpart12asdszdfw}
\sum_{k = 0}^n \sum_{j = 0}^k (- 1)^j n \binom{n}{j} s_{k - j} t_{n - k + j}=\sum_{k=0}^n(-1)^k k\binom{n}{k}\tau_{n-k}(\sigma_k-\sigma_{k-1}),
\end{equation}
or more generally, for $m\geq 0$, the following identity holds:
\begin{equation}\label{eq:2ndpart12asdszdfw_general}
\sum_{k = 0}^n \sum_{j = 0}^k (- 1)^j n \binom{n}{j} s_{m+k - j} t_{n - k + j} = \sum_{k=0}^n \sum_{j=0}^m (-1)^{k+j} k \binom{n}{k}
\binom{m}{j} \tau_{n-k} \left(\sigma_{j+k}-\sigma_{j+k-1}\right).
\end{equation}
\end{theorem}
\begin{proof}
Write $y/x$ for $y$ in~\eqref{hctczju} and use $\binom{n-1}k=\frac{n-k}n\binom nk$ to obtain
\begin{align*}
\sum_{k = 0}^n \sum_{j = 0}^k (- 1)^j \binom{{n}}{j} x^{k - j} y^{n - k + j} &= \sum_{k = 0}^n (- 1)^k \binom{{n - 1}}{k} x^{n - k} y^k \\
&= \sum_{k=0}^n (-1)^k \binom{n}{k} \frac{n-k}{n} x^{n-k} y^k \\
&= \sum_{k=0}^n (-1)^{n-k} \binom{n}{k} \frac{k}{n} x^{k}y^{n-k}.
\end{align*}
Act on both sides with the operator $\mathcal L_x\mathcal L_y$ (see Lemma \ref{main_lem}), which gives the intermediate result
\begin{equation}\label{eq:intermediate_identity}
\sum_{k = 0}^n \sum_{j = 0}^k (- 1)^j n \binom{n}{j} s_{k - j} t_{n - k + j} = \sum_{k = 0}^n (- 1)^{n - k} k \binom{n}{k} s_k t_{n - k}.
\end{equation}

To verify \eqref{eq:2ndpart12asdszdfw} first notice that by the binomial theorem:
$$\sum_{k=0}^n \binom{n}{k} (-1)^k x^k y^{n-k} = (y-x)^n = ((1-x)-(1-y))^n = \sum_{k=0}^n \binom{n}{k} (-1)^{n-k} (1-y)^{n-k}(1-x)^k.$$
Apply the operator $x\cdot \frac{d}{dx}$ to that identity, which yields
$$\sum_{k=0}^n (-1)^k \binom{n}{k} k x^k y^{n-k} = \sum_{k=0}^n (-1)^{n-k} \binom{n}{k} (1-y)^{n-k} k(-x) (1-x)^{k-1}.$$
Write $(-x)(1-x)^{k-1}=(1-x)^{k}-(1-x)^{k-1}$ and multiply through by $(-1)^n$. To finish the proof, apply the operator $\mathcal L_x\mathcal L_y$, use Lemma \ref{main_lem} and combine with \eqref{eq:intermediate_identity}. The identity \eqref{eq:2ndpart12asdszdfw_general} is a consequence of \eqref{s4v4rqn} (see also the proof of Theorem \ref{thm.qtnqdr9}).
\end{proof}

The next results are corollaries to Theorem~\ref{thm.ef1drze}.

\begin{corollary}\label{cor.cp9d6at}
Let $\{s_n,\sigma_n \}$, $n=0,1,2,\ldots$, be a binomial transform pair. Then
\begin{equation}
\sum_{k = 1}^n \sum_{j = 0}^{k - 1} ( - 1)^j n \binom{{n}}{j} \frac{s_{m + n - k + j}}{k - j} = \sum_{k = 0}^m ( - 1)^k \binom{{m}}{k}\sigma _{k + n} - ( - 1)^n s_{m + n}.
\end{equation}
In particular,
\begin{equation}\label{msdi8n1}
\sum_{k = 1}^n \sum_{j = 0}^{k - 1} (- 1)^j n \binom{n}{j} \frac{s_{n - k + j}}{k - j} = \sigma_n - (- 1)^n s_n.
\end{equation}
\end{corollary}
\begin{proof}
Apply Theorem~\ref{thm.ef1drze} with $t_n=1/n$.  
\end{proof}

\begin{corollary}\label{cor.eb6nqko}
Let $m$ and $n$ be non-negative integers. If $(s_n)$ is a self-inverse sequence, then
\begin{equation}
\sum_{k = 1}^n \sum_{j = 0}^{k - 1} (- 1)^j n\binom{{n}}{j} \frac{{s_{m + n - k + j} }}{{k - j}} = \sum_{k = 0}^{m - 1} ( - 1)^k \binom{{m}}{k}s_{k + n} + \left( ( - 1)^m - ( - 1)^n \right) s_{m + n},
\end{equation}
while if $(s_n)$ is an anti-self-inverse sequence, then
\begin{equation}
\sum_{k = 1}^n \sum_{j = 0}^{k - 1} (- 1)^j n\binom{{n}}{j} \frac{{s_{m + n - k + j} }}{{k - j}} = -\sum_{k = 0}^{m - 1} ( - 1)^k \binom{{m}}{k}s_{k + n} - \left( ( - 1)^m + ( - 1)^n \right) s_{m + n}.
\end{equation}
In particular, if $(s_n)$ is a self-inverse sequence, then
\begin{equation}
\sum_{k = 1}^n \sum_{j = 0}^{k - 1} (- 1)^j n \binom{n}{j} \frac{s_{n - k + j}}{k - j}
 = \begin{cases}
 0,& \text{if $n$ is even;} \\ 
 2s_n,& \text{if $n$ is odd;}  \\ 
 \end{cases} 
\end{equation}
while if $(s_n)$ is an anti-self-inverse sequence, then
\begin{equation}
\sum_{k = 1}^n \sum_{j = 0}^{k - 1} (- 1)^j n \binom{n}{j} \frac{s_{n - k + j}}{k - j}
 = \begin{cases}
 -2s_n,& \text{if $n$ is even;} \\ 
 0,& \text{if $n$ is odd.}  \\ 
 \end{cases} 
\end{equation}
\end{corollary}
\begin{proof}
Use the definitions of self-inverse and anti-self-inverse sequences from \eqref{self} and \eqref{anti_self} in Corollary~\ref{cor.cp9d6at}.
\end{proof}

\begin{corollary}\label{cor_Ber_fin}
Let $\{s_n,\sigma_n \}$, $n=0,1,2,\ldots$, be a binomial transform pair. Then
\begin{equation}
\sum_{k = 0}^n \sum_{j = 0}^k (- 1)^j \binom{{n}}{j} s_{m + k - j} = (- 1)^n \sum_{k = 0}^m (- 1)^k \binom{{m}}{k} \left( \sigma_{k + n} - \sigma_{k + n - 1} \right).
\end{equation}
In particular,
\begin{equation}\label{b5uqn54}
\sum_{k = 0}^n \sum_{j = 0}^k (- 1)^j \binom{n}{j} s_{k - j} = (- 1)^n \left( \sigma_n - \sigma_{n - 1} \right).
\end{equation}
\end{corollary}
\begin{proof}
Apply Theorem~\ref{thm.ef1drze} with $t_n=1$ and use the identity~\cite[Equation (5.2)]{Boya4}
\begin{equation*}
\sum_{k = 0}^n \binom{n}{k} (-1)^k k s_{k} = n(\sigma_n - \sigma_{n-1}).
\end{equation*}
\end{proof}

\begin{corollary}
  Let $\{s_n,\sigma_n\}$ be the binomial transform pair and $t$ be the complex number. The following polynomial identity holds:
  \begin{equation}\label{eq:oolyn_s_sigma}
    \sum_{k=0}^n \sum_{j=0}^k (-1)^j n \binom{n}{j} t^{k-j} s_{n-k+j} = \sum_{k=1}^n (-1)^{k+1} k \binom{n}{k} \sigma_{n-k} t(1-t)^{k-1}.
  \end{equation}
\end{corollary}
\begin{proof}
  Apply Theorem~\ref{thm.ef1drze} with $t_n=t^n$.
\end{proof}
The identity \eqref{eq:oolyn_s_sigma} can be used to produce the following result.

\begin{theorem}\label{thrm:btp_harmonic}
Let $\{s_n,\sigma_n \}$ be the binomial transform pair. The following identities hold for $r,s\in\mathbb{C}\setminus\mathbb{Z}^-$ with $s\neq 0$ and $r-s\notin\mathbb{Z}^-$:
  \begin{equation}\label{eq:ds_binom_generalss}
  \sum_{k=0}^n\sum_{j=0}^k\frac{(-1)^j n\binom{n}{j}}{(k-j+s)\binom{k-j+r}{k-j+s}}s_{n-k+j}=\sum_{k=1}^n\frac{(-1)^{k+1}k\binom{n}{k}}{(s+1)\binom{r+k}{s+1}}\sigma_{n-k}.
  \end{equation}
  and
  \begin{equation}\label{eq:ds_binom_generalss_b}
  \sum_{k=0}^n\sum_{j=0}^k\frac{(-1)^j n\binom{n}{j}(H_{r-s}-H_{k-j+r})}{(k-j+s)\binom{k-j+r}{k-j+s}}s_{n-k+j}=\sum_{k=1}^n\frac{(-1)^{k+1}k\binom{n}{k}(H_{r+k-s-1}-H_{r+k})}{(s+1)\binom{r+k}{s+1}}\sigma_{n-k}.
  \end{equation} 
\end{theorem}
\begin{proof}
    Multiply \eqref{eq:oolyn_s_sigma} by $(1-t)^rt^{s-1}$ and integrate from $0$ to $1$ with respect to $t$ to obtain
    $$ \sum_{k=0}^n\sum_{j=0}^k(-1)^jn\binom{n}{j}s_{n-k+j}\int_0^1t^{k-j+s-1}(1-t)^{r}dt=\sum_{k=1}^n(-1)^{k+1}k\binom{n}{k}\sigma_{n-k}\int_0^1t^s(1-t)^{r+k-1}dt.$$
    Use the following integral (which is the value of the Beta integral):
    $$\int_0^1 x^u(1-x)^vdx=\frac{1}{(u+1)\binom{u+v+1}{u+1}},$$
    which holds for $\Re (u)>-1$ and $\Re(v)>-1$. This gives
    $$\sum_{k=0}^n\sum_{j=0}^k\frac{(-1)^j n\binom{n}{j}}{(k-j+s)\binom{k-j+r+s}{k-j+s}}s_{n-k+j}=\sum_{k=1}^n\frac{(-1)^{k+1}k\binom{n}{k}}{(s+1)\binom{r+s+k}{s+1}}\sigma_{n-k}.$$
    Replacing $r\mapsto r-s$ gives \eqref{eq:ds_binom_generalss}. Identity \eqref{eq:ds_binom_generalss_b} is obtained from the former via differentiation with respect to $r$ and by taking note that 
    $$\frac{d}{dx}\binom{x+a}{y}^{-1}=\binom{x+a}{y}^{-1}(H_{a+x-y}-H_{x+a}).$$
\end{proof}

\begin{corollary}
Let $\{s_n,\sigma_n \}$ be the binomial transform pair. Then
\begin{equation}\label{eq:duble_harmonic_ssigma_sdv3uiwrh}
\sum_{k=0}^n\sum_{j=0}^k\frac{(-1)^{j}n\binom{n}{j}}{k-j+1}H_{k-j+1}s_{n-k+j}=\sum_{k=0}^n(-1)^{k}\binom{n}{k}\frac{2k+1}{k(k+1)^2}\sigma_{n-k}.
\end{equation}
\end{corollary}
\begin{proof}
Set $r=s=1$ in \eqref{eq:ds_binom_generalss_b}.
\end{proof}

\begin{remark}
  We note that a more general version of Theorem \ref{thrm:btp_harmonic} can be stated in a fashion similar to Theorem \ref{thm.ef1drze}.
\end{remark}

A different direction can also be used to derive further results.

\begin{corollary}\label{int_cor}
    The following identity holds for $n>0$
    \begin{equation}
        \sum_{k=0}^n\sum_{j=0}^k(-1)^{j+1}\binom{n+1}{j+1} (1-y)^{j+1} = (n+1)y^n + y^{n+1} - (n+2)
    \end{equation}
    from which it follows that
    \begin{equation}
       \sum_{k=0}^n\sum_{j=0}^k (-1)^{j+1} \binom{n+1}{j+1} s_{j+1} = (n+1)\sigma_n + \sigma_{n+1} - (n+2)\sigma_0.
    \end{equation}
\end{corollary}
\begin{proof}
  Take an indefinite integral of the identity \eqref{gr3jv5j}. Then, after rearranging we gain
  $$\sum_{k=0}^n\sum_{j=0}^k(-1)^{j+1}\binom{n+1}{j+1}(1-y)^{j+1}=(n+1)y^n+y^{n+1}+C.$$
  To obtain the value of $C$ set $y=1$. This gives the first identity. The second identity is an application of Lemma \ref{main_lem}. 
	Take note that $n+2=(n+2)y^0.$
\end{proof}

Further integration leads to further identities.

\begin{corollary}
    The following identity holds
    \begin{equation}
    \sum_{k=0}^n \sum_{j=0}^k (-1)^{j} \binom{n+2}{j+2} (1-y)^{j+2} = (n+2)y^{n+1}+y^{n+2}-(n+2)^2y+n^2+3n+1
    \end{equation}
    from which it follows that
    \begin{equation}
    \sum_{k=0}^n \sum_{j=0}^k (-1)^{j} \binom{n+2}{j+2} s_{j+2} = (n+2)\sigma_{n+1}+\sigma_{n+2}+(n+2)^2s_1-(n+3)s_0.
    \end{equation}
\end{corollary}

\section{Applications involving special numbers}

Here we present a collection of double sum identities that can be deduced from the main results from the last section.
Each identity is stated as a proposition. 

\subsection{Identities involving Bernoulli numbers}

\begin{proposition}
If $n$ is a positive integer, then
\begin{equation}
\sum_{k = 0}^n (- 1)^k \sum_{j = 0}^k \binom{n}{j} B_{n + j - k} = - B_{n - 1}.
\end{equation}
\end{proposition}
\begin{proof}
From \eqref{Ber_bt} it is clear that we have the binomial transform pair $\{(-1)^n B_n,(-1)^n B_n\}$. 
Now, a simple application of Theorem \ref{thm.qtnqdr9} gives the result.
\end{proof}

\begin{proposition}
If $n$ is a positive integer, then
\begin{equation}
\sum_{k = 0}^n \sum_{j = 0}^k \binom{n}{j} (- 1)^k 2^{n-k+j} B_{n-k+j} = (-1)^n (2 - 2^{n-1}) B_{n - 1}.
\end{equation}
In particular, if $n$ is even, then
\begin{equation}
\sum_{k = 0}^n \sum_{j = 0}^k \binom{n}{j} (- 1)^k 2^{n-k+j} B_{n-k+j} = 0.
\end{equation}
\end{proposition}
\begin{proof}
Consider the identity
\begin{equation*}
\sum_{k = 0}^n {\binom nky^k B_k }  = y^n B_n\left(\frac1y\right) 
\end{equation*}
which is~\eqref{nskd3u8} with $1/y$ written for $x$. Setting $y=2$ and using known evaluation $B_k(1/2)=(2^{1-k}-1)B_k$ shows that we have the binomial transform pair $\{(-1)^n 2^n B_n,(2 - 2^{n}) B_{n}\}$, 
and an application of Theorem \ref{thm.qtnqdr9} yields the result. The second statement follows as $B_{2n+1}=0$ for all $n\geq 1$.
\end{proof}

\begin{remark}
Values of $B_k(1/3)$, $B_k(2/3)$, $B_k(1/4)$, $B_k(3/4)$, $B_k(1/6)$ and $B_k(5/6)$ are also known for $k$ even~\cite{Camargo21}. They can be used to formulate more identities with Bernoulli numbers. These are left for the interested reader.
\end{remark}

\begin{proposition}
If $n$ is a positive integer, then
\begin{equation}
\sum_{k = 0}^n \sum_{j = 0}^k \binom{n}{j} (- 1)^j 2^{n-k+j} B_{n-k+j} = (-1)^{n-1} (2 - 2^{n-1}) B_{n - 1} + 2(n-1).
\end{equation}
In particular, if $n$ is even, then
\begin{equation}
\sum_{k = 0}^n \sum_{j = 0}^k \binom{n}{j} (- 1)^j 2^{n-k+j} B_{n-k+j} = 2(n-1).
\end{equation}
\end{proposition}
\begin{proof}
Work with
\begin{equation*}
\sum_{k = 0}^n \binom{n}{k} (-1)^k 2^{k} B_{k} = (-1)^n (2 - 2^{n}) B_{n} + 2n, \quad n\geq 0,
\end{equation*}
which can be shown without effort. This shows that we have the binomial transform pair $\{2^n B_n,(-1)^n (2 - 2^{n}) B_{n}+2n\}$, 
and an application of Theorem \ref{thm.qtnqdr9} yields the result.
\end{proof}

\begin{proposition}
If $n$ is a positive integer, then
\begin{equation}
\sum_{k = 0}^n \sum_{j = 0}^k \binom{n}{j} (- 1)^k \left (2^{n-k+j+1} - 1\right ) \frac{B_{n-k+j+1}}{n-k+j+1} = (1 - 2^{n})\frac{B_{n}}{n}.
\end{equation}
\end{proposition}
\begin{proof}
Work with the binomial transform pair 
$$\{(-1)^{n+1}(2^{n+1}-1) B_{n+1}/(n+1),(-1)^{n+1}(2^{n+1}-1) B_{n+1}/(n+1)\},$$ 
and apply Theorem \ref{thm.qtnqdr9}.
\end{proof}

\begin{proposition}
If $n$ is a positive integer, then
\begin{equation}\label{vrhyzb2}
\sum_{k = 0}^n (- 1)^k \sum_{j = 0}^k \binom{{n}}{j} (n + j - k) B_{n + j - k - 1} = (n - 1)B_{n - 2}, 
\end{equation}
and
\begin{equation}\label{zir34uy}
\sum_{k = 0}^n (- 1)^k \sum_{j = 0}^k \binom{{n}}{j} (n + j - k) B_{n + j - k}
 =   
 \begin{cases}
 - (n - 1)B_{n - 1},&\text{if $n$ is odd;}  \\ 
 - (n - 1)B_{n - 2},&\text{if $n>2$ and $n$ is even;}  \\
 -1/2,&\text{if $n=2$.} 
 \end{cases} 
\end{equation}
\end{proposition}
\begin{proof}
The following identities are known~\cite[Sections 7, 8]{Adegoke0}:
\begin{equation}\label{hmg5qf4}
\sum_{k = 0}^n {\binom nkkB_k }  = ( - 1)^n n\left( {B_n  + B_{n - 1} } \right)
\end{equation}
and
\begin{equation}\label{tv4r72c}
\sum_{k = 1}^n {\binom{{n}}{k}kB_{k - 1} }  = ( - 1)^{n - 1} nB_{n - 1} .
\end{equation}
Use of~\eqref{tv4r72c} in~\eqref{utju7bn} gives~\eqref{vrhyzb2} while use of~\eqref{hmg5qf4} produces~\eqref{zir34uy}.
\end{proof}

\begin{proposition}
If $n$ is a positive integer, then
\begin{equation}
\sum_{k = 1}^n \sum_{j = 0}^{k - 1} \binom{n}{j} (-1)^k \frac{B_{n-k+j}}{k-j}
 = \begin{cases}
 -1,& \text{if $n=1$} \\ 
 0,& \text{if $n\geq 2$;}  \\ 
 \end{cases} 
\end{equation}
and
\begin{equation}
\sum_{k = 1}^n \sum_{j = 0}^{k - 1} \binom{n}{j} (-1)^k \left (2^{n-k+j+1} - 1 \right) \frac{B_{n-k+j+1}}{(n-k+j+1)(k-j)}
 = \begin{cases}
 0,& \text{if $n$ is even} \\ 
 2(2^{n+1}-1)\frac{B_{n+1}}{n(n+1)},& \text{if $n$ is odd.} \\
 \end{cases} 
\end{equation}
\end{proposition}
\begin{proof}
The sequences $s_n=(-1)^n B_n$ and $s_n=(-1)^{n+1}(2^{n+1}-1)B_{n+1}/(n+1)$ are self-inverse sequences. 
The identities are consequences of Corollary \ref{cor.cp9d6at}. 
\end{proof}

\begin{proposition}
If $n$ is a positive integer, then
\begin{equation}
\sum_{k = 0}^n \sum_{j = 0}^k \binom{n}{j} (-1)^k B_{k-j} = B_n + B_{n-1},
\end{equation}
and
\begin{equation}
\sum_{k = 0}^n \sum_{j = 0}^k \binom{n}{j} (-1)^k 2^{k-j} B_{k-j} = (-1)^n ((2-2^n)B_n - (2-2^{n-1})B_{n-1}). 
\end{equation}
\end{proposition}
\begin{proof}
Apply Corollary \ref{cor_Ber_fin} to the sequences $s_n=(-1)^n B_n$ and $s_n=(-1)^{n} 2^{n} B_{n}$, respectively.  
\end{proof}

\subsection{Identities involving Fibonacci numbers}

\begin{lemma}
If $m$ and $n$ are non-negative integers and $r$ and $t$ are integers, then
\begin{equation}\label{eq:4.2_first_identity}
\sum_{k = m}^n {( - 1)^k \binom{{n}}{k}\binom{{k}}{m}\frac{{G_{r + (k - m)t} }}{{L_t^k }}}  = ( - 1)^{(n - m)t + m} \binom{{n}}{m}\frac{{G_{r - (n - m)t} }}{{L_t^n }}.
\end{equation}
\end{lemma}
\begin{proof}
Differentiating the binomial identity
$$\sum_{k=0}^n(-1)^k\binom{n}{k}x^ky^{n-k}=(y-x)^n$$
$m$ times with respect to $x$ and dividing by $m!$ yields
$$\sum_{k=m}^n(-1)^k\binom{n}{k}\binom{k}{m}x^{k-m}y^{n-k}=(-1)^m\binom{n}{m}(y-x)^{n-m}.$$
Substitute $x=\alpha^t$ and $y=\alpha^t+\beta^{t}$, which yields
$$\sum_{k=m}^{n}(-1)^k\binom{n}{k}\binom{k}{m}\frac{\alpha^{(k-m)t}}{L_t^k}=(-1)^m\binom{n}{m}\frac{\beta^{(n-m)t}}{L_t^n}=(-1)^{m+(n-m)t}\binom{n}{m}\frac{\alpha^{-(n-m)t}}{L_t^n}.$$
Similar expression obtained with substitution $x=\beta^t$ and $y=\alpha^t+\beta^t$ together with the Binet formula completes the proof of \eqref{eq:4.2_first_identity}.
\end{proof}

\begin{proposition}
If $m$ and $n$ are non-negative integers and $r$ and $t$ are integers, then
\begin{equation}\label{eeyr1j1}
\sum_{k = 0}^n \sum_{j = 0}^k (- 1)^j \binom{{n}}{j} \binom{{n - k + j}}{m} L_t^{k - j} G_{r + (n - k + j - m)t}
= (- 1)^{(n - m - 1)t + m} \binom{{n - 1}}{m} L_t G_{r - (n - m - 1)t}.
\end{equation}
In particular,
\begin{align}
\sum_{k = 0}^n (- 1)^k \sum_{j = 0}^k \binom{{n}}{{k - j}} F_j &= (- 1)^n F_{2n - 1},\label{ibhbsoo} \\
\sum_{k = 0}^n (- 1)^k \sum_{j = 0}^k \binom{{n}}{{k - j}} L_j &= (- 1)^n L_{2n - 1}\label{n5ooj8u}. 
\end{align}
\end{proposition}
\begin{proof}
From~\eqref{eq:4.2_first_identity}, we identify the binomial transform pair $(s_k)$ and $(\sigma_k)$, where
\begin{equation*}
s_k  = \binom{{k}}{m}\frac{{G_{r + (k - m)t} }}{{L_t^k }}\text{ and }\sigma _k  = ( - 1)^{(k - m)t + m} \binom{{k}}{m}\frac{{G_{r - (k - m)t} }}{{L_t^k }}.
\end{equation*}
Use these in~\eqref{utju7bn} to obtain~\eqref{eeyr1j1}. Identities~\eqref{ibhbsoo} and~\eqref{n5ooj8u} are special cases of~\eqref{eeyr1j1} obtained at $m=0$, $t=1$ and $r=-n$.
\end{proof}

\begin{proposition}
If $n$ is a non-negative integer and $r$ is an integer, then
\begin{align}\label{jvjrby6}
&\sum_{k = 0}^n {\sum_{j = 0}^k {( - 1)^j \binom nj\binom {n-k+j}mF_{n + j + r - k - m}^3 } }\nonumber\\
&\qquad =\frac{(-1)^m}5 \binom{n-1}m\left( {( - 1)^{n - m -1} 2^{n - m - 1} F_{n + 3r - m - 1}  - ( - 1)^r 3F_{2(n - m - 1) + r} } \right)
\end{align}
and
\begin{align}\label{oxhs032}
&\sum_{k = 0}^n {\sum_{j = 0}^k {( - 1)^j \binom nj\binom {n-k+j}mL_{n + j + r - k - m}^3 } }\nonumber\\
&\qquad =(-1)^m\binom{n-1}m\left( {( - 1)^{n - m -1} 2^{n - m - 1} L_{n + 3r - m - 1}  + ( - 1)^r 3L_{2(n - m - 1) + r} } \right)
\end{align}
In particular,
\begin{equation}
\sum_{k = 0}^n {( - 1)^k \sum_{j = 0}^k {\binom{{n}}{j}\binom{{n - k + j}}{m}F_{n + k - m - j - 2}^3 } }  = \frac{{( - 1)^n }}{5}\binom{{n - 1}}{m}2^{n - m - 1} F_{5(n - m - 1)} ,
\end{equation}
with the special value
\begin{equation}
\sum_{k = 0}^n {( - 1)^k \sum_{j = 0}^k {\binom{{n}}{j}\binom{{n - k + j}}{{n - 2}}F_{k - j}^3 } }  = ( - 1)^n 2\left( {n - 1} \right).
\end{equation}
\end{proposition}
\begin{proof}
From the known identities~\cite{Adegoke21}:
\begin{equation*}
\sum_{k = 0}^n {( - 1)^k \binom{{n}}{k}F_{k + r}^3 }  = \frac{1}{5}\left( {( - 1)^n 2^n F_{n + 3r}  - ( - 1)^r 3F_{2n + r} } \right)
\end{equation*}
and
\begin{equation*}
\sum_{k = 0}^n {( - 1)^k \binom{{n}}{k}L_{k + r}^3 }  = ( - 1)^n 2^n L_{n + 3r}  + ( - 1)^r 3L_{2n + r} ,
\end{equation*}
we recognize the binomial transform pairs $\{a_k,\alpha_k\}$, where
\begin{equation}\label{glq7t03}
a_k  = F_{k + r}^3 \text{ and }\alpha _k  = \frac{1}{5}\left( {( - 1)^k 2^k F_{k + 3r}  - ( - 1)^r 3F_{2k + r} } \right)
\end{equation}
and also,
\begin{equation}\label{piuu548}
a_k  = L_{k + r}^3 \text{ and }\alpha _k  = ( - 1)^k 2^k L_{k + 3r}  + ( - 1)^r 3L_{2k + r}.
\end{equation}
This also means (see~\cite[section 7]{Adegoke0}) that for $m=0,1,2,\ldots$,
\begin{equation}\label{gs5dmhf}
s_k  = \binom km F_{k - m + r}^3 \text{ and }\sigma _k  = \frac{(-1)^m}{5}\binom km\left( {( - 1)^{k-m} 2^{k-m} F_{k + 3r-m}  - ( - 1)^r 3F_{2k -2m + r} } \right)
\end{equation}
and also,
\begin{equation}\label{vm8tp3m}
s_k  = \binom km L_{k - m + r}^3 \text{ and }\sigma _k  =(-1)^m\binom km\left( ( - 1)^{k-m} 2^{k-m} L_{k -m+ 3r}  + ( - 1)^r 3L_{2k -2m+ r}\right),
\end{equation}
are each a binomial transform pair.

Use of~\eqref{gs5dmhf} and~\eqref{vm8tp3m}, in turn, in~\eqref{utju7bn} yield~\eqref{jvjrby6} and~\eqref{oxhs032}
\end{proof}

\begin{proposition}
If $n$ is a non-negative integer and $r$ is an integer, then
\begin{equation}
\sum_{k = 1}^n \sum_{j = 0}^{k - 1} (- 1)^j n \binom{{n}}{j} \frac{F_{n - k + j + r}^3}{k - j}
= \frac{1}{5} \left((- 1)^n 2^n F_{n + 3r} - (- 1)^r 3 F_{2n + r} \right) - (- 1)^n F_{n + r}^3 
\end{equation}
and
\begin{equation}
\sum_{k = 1}^n \sum_{j = 0}^{k - 1} (- 1)^j n \binom{{n}}{j} \frac{L_{n - k + j + r}^3}{k - j} 
= (- 1)^n 2^n L_{n + 3r} + (- 1)^r 3 L_{2n + r} - (- 1)^n L_{n + r}^3 .
\end{equation}
In particular,
\begin{equation}
\sum_{k = 1}^n (- 1)^{k - 1} \sum_{j = 0}^{k - 1} n \binom{{n}}{j} \frac{F_{k - j}^3}{k - j} = \frac{(-1)^{n+1}}{5} \left( 2^n F_{2n} + 3F_n \right)
\end{equation}
and
\begin{equation}
\sum_{k = 1}^n (- 1)^{n - k} \sum_{j = 0}^{k - 1} n \binom{{n}}{j} \frac{L_{k - j}^3}{k - j} = 2^n L_{2n} + 3L_n - 8.
\end{equation}
\end{proposition}
\begin{proof}
Use~\eqref{glq7t03} and~\eqref{piuu548}, in turn, in~\eqref{msdi8n1}.
\end{proof}

\begin{proposition}
If $n$ is a non-negative integer, then
\begin{equation}
\sum_{k = 1}^n {\sum_{j = 0}^{k - 1} {( - 1)^j n\binom{{n}}{j}\frac{{L_{n - k + j} }}{{k - j}}} }
 = \begin{cases}
 0,&\text{if $n$ is even;} \\ 
 2L_n,&\text{if $n$ is odd;}  \\ 
 \end{cases} 
\end{equation}
and
\begin{equation}
\sum_{k = 1}^n {\sum_{j = 0}^{k - 1} {( - 1)^j n\binom{{n}}{j}\frac{{F_{n - k + j} }}{{k - j}}} }
 = \begin{cases}
 -2F_n,&\text{if $n$ is even;} \\ 
 0,&\text{if $n$ is odd.}  \\ 
 \end{cases} 
\end{equation}
\end{proposition}

\begin{proposition}
If $n$ is a non-negative integer and $r$ is an integer, then
\begin{equation}
\sum_{k = 0}^n \sum_{j = 0}^k ( - 1)^j \binom{{n}}{j} G_{r + k - j} = G_{r - n + 2}.
\end{equation}
In particular,
\begin{equation}
\sum_{k = 0}^n \sum_{j = 0}^k ( - 1)^j \binom{{n}}{j} F_{k - j} = ( - 1)^{n - 1} F_{n - 2} 
\end{equation}
and
\begin{equation}
\sum_{k = 0}^n \sum_{j = 0}^k ( - 1)^j \binom{{n}}{j} L_{k - j} = ( - 1)^n L_{n - 2}.
\end{equation}
\end{proposition}

\begin{proof}
Set $m=0$ and $t=1$ in~\eqref{eq:4.2_first_identity} to get the binomial transform pair $(s_k)$ and $(\sigma_k)$, $k=0,1,2,\ldots$, where
\begin{equation*}
s_k=G_{r+k}\text{ and }\sigma_k=(-1)^kG_{r-k}. 
\end{equation*}
Use these in~\eqref{b5uqn54}.
\end{proof}

\subsection{Identities involving harmonic numbers}

\begin{proposition}
If $n$ is a positive integer and $r$ is a complex number that is not a negative integer, then
\begin{equation}
\sum_{k = 0}^n {\sum_{j = 0}^k {( - 1)^j \binom{{n}}{j}H_{n + r + j - k} } }
= \begin{cases}
 H_r,&\text{if $n=1$;}  \\ 
  - \dfrac{1}{{n - 1}}\binom{{n - 1 + r}}r^{ - 1},&\text{if $n>1$;}  \\ 
 \end{cases} 
\end{equation}
and
\begin{equation}
\sum_{k = 0}^n {\sum_{j = 0}^k {( - 1)^j \binom{{n}}{j}O_{n + j - k} } }  =  - \frac{1}{{n - 1}}\binom{{2(n - 1)}}{{n - 1}}^{ - 1} 2^{2n - 3} .
\end{equation}
\end{proposition}
\begin{proof}
From~\cite[Equation (9.37)]{Boya2}:
\begin{equation}
\sum_{k = 0}^n {( - 1)^k \binom{{n}}{k}H_{k + r} }  =  - \frac{1}{n}\binom{{n + r}}r^{ - 1},\quad n\ne0,\label{eq.wt34kp1}
\end{equation}
and~\cite[Equation (9.51)]{Boya2}:
\begin{equation}\label{eq.jupxl7y}
\sum_{k = 0}^n {( - 1)^k \binom{{n}}{k}O_k }  =  - \binom{{2n}}{n}^{ - 1} \frac{{2^{2n - 1} }}{n},
\end{equation}
we can identify the binomial-transform pair $(s_n)$ and $(\sigma_n)$, $n=0,1,2,\ldots$, where
\begin{equation}\label{gxb8nr4}
s_n  = H_{n + r} ,\quad \sigma_n  =  \frac{{\delta _{n0}(1+H_r)-1 }}{{n + \delta _{n0} }}\binom{{n + r}}r^{ - 1},  
\end{equation}
and also $(t_n)$ and $(\tau_n)$, $n=0,1,2,\ldots$, where
\begin{equation}\label{h5ajier}
t_n  = O_n ,\quad \tau_n = - \frac{{1 - \delta _{n0} }}{{n + \delta_{n0}}}\binom{{2n}}{n}^{ - 1} 2^{2n - 1}. 
\end{equation}
Here and throughout this paper, $\delta_{ij}$ is Kronecker's delta having the value $1$ when $i=j$ and zero otherwise.
The results follow upon using~\eqref{gxb8nr4} and~\eqref{h5ajier} in Theorem~\ref{thm.qtnqdr9}.
\end{proof}

\begin{proposition}
If $n$ is a positive integer, then
\begin{equation}
\sum_{k = 0}^n {\sum_{j = 0}^k {( - 1)^j \binom{{n}}{j}H_{k + r - j} } } = (-1)^n
 \begin{cases}
 -H_{r+1},&\text{if $n=1$;}  \\ 
 \frac{1}{{n - 1}}\binom{{n + r - 1}}{r}^{ - 1} - \frac{1}{n}\binom{{n + r}}{r}^{ - 1} ,&\text{if $n>1$.} 
 \end{cases} 
\end{equation}
\end{proposition}
\begin{proof}
Use~\eqref{gxb8nr4} in~\eqref{b5uqn54}.
\end{proof}
\begin{lemma}[{\cite[Section 13]{Adegoke0}}]
If $n$ is a non-negative integer, then
\begin{equation}\label{ey4kpm3}
\sum_{k = 0}^n {( - 1)^k \binom{{n}}{k}2^{ - k} \binom{{2k}}{k}\left( {H_k - O_k } \right)}
 = \begin{cases}
 2^{ - n - 1} \binom{{n}}{{n/2}}H_{n/2},&\text{if $n$ is even;}  \\ 
 0,&\text{if $n$ is odd.} \\ 
 \end{cases} 
\end{equation}
\end{lemma}

\begin{proposition}
If $n$ is a non-negative integer, then
\begin{align}
&\sum_{k = 0}^n {\sum_{j = 0}^k {( - 1)^j \binom{{n}}{j}2^{k - j} \binom{{2\left( {n - k + j} \right)}}{{n - k + j}}\left( {H_{n - k + j}  - O_{n - k + j} } \right)} }\nonumber\\
&\qquad =\begin{cases}
 \binom{{n - 1}}{{\left( {n - 1} \right)/2}}H_{\left( {n - 1} \right)/2},&\text{if $n$ is odd;}  \\ 
 0,&\text{if $n$ is even.} \\ 
 \end{cases} 
\end{align}
\end{proposition}
\begin{proof}
Use~\eqref{ey4kpm3} in Theorem~\ref{thm.qtnqdr9}.
\end{proof}

\begin{proposition}
If $n$ is a non-negative integer, then
\begin{align}
&\sum_{k = 1}^n {\sum_{j = 0}^{k - 1} {( - 1)^j n\binom{{n}}{j}2^{k - j} \binom{{2\left( {n - k + j} \right)}}{{n - k + j}}\frac{{\left( {H_{n - k + j}  - O_{n - k + j} } \right)}}{{k - j}}} } \nonumber\\
&\qquad = ( - 1)^{n + 1} \binom{{2n}}{n}\left( {H_n  - O_n } \right) +
 \begin{cases}
 \frac{1}{2}\binom{{n}}{{n/2}}H_{n/2},&\text{if $n$ is even;}  \\ 
 0,&\text{if $n$ is odd.} \\ 
 \end{cases} 
\end{align}
\end{proposition}
\begin{proof}
Use~\eqref{ey4kpm3} in Corollary~\ref{cor.cp9d6at}.
\end{proof}

\begin{lemma}[{\cite[Section 13]{Adegoke0}}]\label{lem.sr1lpuj}
If $n$ is a non-negative integer, then
\begin{equation}
\sum_{k = 0}^n {( - 1)^k \binom{{n}}{k}\frac{{2^k }}{{\left( {k + 1} \right)^2 }}}  = O_{\left\lfloor {(n + 2)/2} \right\rfloor } 
\end{equation}
and
\begin{align}\label{y56cnyf}
\sum_{k = 0}^n {( - 1)^k \binom{{n}}{k}\frac{{2^k }}{{\left( {k + 1} \right)^3 }}}  &= \frac{1}{4}\frac{{H_{n + 1}^2  + H_{n + 1}^{(2)} }}{{n + 1}} - \frac{1}{8}\frac{{H_{\left\lceil {n/2} \right\rceil }^2  + H_{\left\lceil {n/2} \right\rceil }^{(2)} }}{{n + 1}} + \frac{1}{2}\frac{{O_{\left\lfloor {(n + 2)/2} \right\rfloor }^2  + O_{\left\lfloor {(n + 2)/2} \right\rfloor }^{(2)} }}{{n + 1}}\nonumber\\
&\qquad - \frac{1}{{2\left( {n + 1} \right)}}\sum_{k = 0}^n {( - 1)^k \frac{{H_{k + 1} }}{{k + 1}}} .
\end{align}
\end{lemma}

\begin{proposition}
If $n$ is a positive integer and $m$ is a non-negative integer, then
\begin{equation}
\sum_{k = 0}^n {\sum_{j = 0}^k {( - 1)^j \binom{{n}}{j}\frac{{2^{m + n + j - k} }}{{\left( {m + n + j - k + 1} \right)^2 }}} }  = \sum_{k = 0}^m {( - 1)^k \binom{{m}}{k}O_{\left\lceil {(n + k)/2} \right\rceil } } 
\end{equation}
and
\begin{align}
\sum_{k = 0}^n {\sum_{j = 0}^k {( - 1)^j \binom{{n}}{j}\frac{{2^{n + j - k} }}{{\left( {n + j - k + 1} \right)^3 }}} }  &= \frac{1}{4}\frac{{H_n^2  + H_n^{(2)} }}{n} - \frac{1}{8}\frac{{H_{\left\lfloor {n/2} \right\rfloor }^2  + H_{\left\lfloor {n/2} \right\rfloor }^{(2)} }}{n} + \frac{1}{2}\frac{{O_{\left\lceil {n/2} \right\rceil }^2  + O_{\left\lceil {n/2} \right\rceil }^{(2)} }}{n}\nonumber\\
&\qquad+ \frac{1}{{2n}}\sum\limits_{k = 1}^n {( - 1)^k \frac{{H_k }}{k}}.
\end{align}
\end{proposition}
\begin{proof}
Use~Lemma~\ref{lem.sr1lpuj} and Theorem~\ref{thm.qtnqdr9}.
\end{proof}

\begin{proposition}
If $n$ is a positive integer, then
\begin{align}
\sum_{k = 0}^n \sum_{j = 0}^k (- 1)^j \binom{n}{j} H_{n + j - k} &= - \frac{1}{{n - 1}},\quad n\ne 1 \label{jw8f41m},\\
\sum_{k = 0}^n \sum_{j = 0}^k (- 1)^j \binom{n}{j} H_{n + j - k}^2 &= \frac{H_{n - 1}}{n - 1} - \frac{2}{(n-1)^2},\quad n\ne 1\label{qwgdu5z},
\end{align}
and
\begin{equation}\label{gt3kuk2}
\sum_{k = 0}^n (- 1)^k \sum_{j = 0}^k \binom{n}{j} H_{n + j - k} = (-1)^n 2^{n - 1} \left(H_{n - 1} - \sum_{k = 1}^{n - 1} \frac{1}{k2^k} \right).
\end{equation}
\end{proposition}
\begin{proof}
Use the binomial transform pair $s_n=1$ and $\sigma_n=\delta_{n0}$ in Corollary~\ref{cor.gca1g47} to obtain~\eqref{jw8f41m}. 
The binomial transform pair $\{s_n,\sigma_n\}$ where
\begin{equation*}
s_n = H_n \qquad \text{and}\qquad 
\sigma_n 
=\begin{cases}
  - \dfrac{1}{n},&\text{if $n\ne 0$;} \\ 
 0,&\text{if $n=0$;} \\ 
 \end{cases}
\end{equation*}
plugged into~\eqref{vsjufnd} gives
\begin{equation*}
\sum_{k = 0}^n \sum_{j = 0}^k ( - 1)^j \binom{{n}}{j} H_{n + j - k}^2 = -\frac{H_{n - 1}}{n - 1} + \sum_{k = 1}^{n - 2} \frac{1}{{k\left( {n - 1 - k} \right)}},
\end{equation*}
from which~\eqref{qwgdu5z} follows since
\begin{equation*}
\frac{1}{{k\left( {n - 1 - k} \right)}} = \frac{1}{{k\left( {n - 1} \right)}} + \frac{1}{{\left( {n - 1 - k} \right)\left( {n - 1} \right)}};
\end{equation*}
so that
\begin{equation*}
\sum_{k = 1}^{n - 2} {\frac{1}{{k\left( {n - 1 - k} \right)}}}  = \frac{{2H_{n - 2} }}{{n - 1}} = \frac{{2H_{n - 1} }}{{n - 1}} - \frac{2}{{\left( {n - 1} \right)^2 }}.
\end{equation*}
Identity~\eqref{gt3kuk2} is obtained by using the binomial transform pair $s_k=(-1)^k$ and $\sigma_k=2^k$ in~\eqref{vsjufnd}.
\end{proof}

\begin{proposition}
If $n$ is a positive integer, then
    $$\sum_{k=0}^n(-1)^k \sum_{j=0}^k \frac{n \binom{n}{j}}{k-j+1} H_{k-j+1} = \sum_{k=0}^n \binom{n}{k} \frac{(2k+1)(-2)^{n-k}}{k(k+1)^2}.$$
\end{proposition}
\begin{proof}
Use the binomial transform pair $s_n=(-1)^n$ and $\sigma_n=2^n$ in \eqref{eq:duble_harmonic_ssigma_sdv3uiwrh}.
\end{proof}

\begin{proposition}
If $m$ is a positive integer, then
\begin{equation}
\sum_{n = 0}^m {\sum_{k = 0}^n {\sum_{j = 0}^k {( - 1)^j \binom{{n}}{j}H_{n + j - k} } } }  =  - H_{m - 1} 
\end{equation}
and
\begin{equation}
\sum_{n = 0}^m {\sum_{k = 0}^n {\sum_{j = 0}^k {( - 1)^j \binom{{n}}{j}H_{n + j - k}^2 } } }  = \frac{1}{2}\left( {H_{m - 1}^2  - 3H_{m - 1}^{(2)} } \right).
\end{equation}
\end{proposition}
\begin{proof}
Sum~\eqref{jw8f41m} and~\eqref{qwgdu5z} from $n=2$ to $n=m$.
\end{proof}

\subsection{Identities involving Catalan numbers and harmonic numbers}

\begin{lemma}
If $n$ is a non-negative integer, then
\begin{align}
\sum_{k = 0}^n {( - 1)^k \binom{{n}}{k}2^{ - k} \binom{{k}}{{\left\lfloor {k/2} \right\rfloor }}}  &= 2^{ - n} C_n \label{w8utb12},\\
\sum_{k = 0}^n {( - 1)^k \binom{{n}}{k}2^{ - k} C_{k + 1} }  
&=\begin{cases}
 0 ,&\text{if $n$ is odd;}\\ 
 2^{ - n} C_{n/2},&\text{if $n$ is even.}  
 \end{cases}\label{i1hknxb}
\end{align}
\end{lemma}
\begin{proof}
Identity~\eqref{w8utb12} is from Mathworld~\cite{mathworld} while~\eqref{i1hknxb} is due to Suleiman and Sury~\cite{sulei23}.
\end{proof}

\begin{proposition}
If $n$ is a positive integer, then
\begin{equation}
\sum_{k = 0}^n {\sum_{j = 0}^k {( - 1)^j \binom{{n}}{j}2^{ - (m - k + j)} C_{m + n - k + j} } }  = \sum_{k = 0}^m {( - 1)^k \binom{{m}}{k}2^{ - (k - 1)} \binom{{n + k - 1}}{{\left\lfloor {\left( {n + k - 1} \right)/2} \right\rfloor }}} 
\end{equation}
and
\begin{equation}
\sum_{k = 0}^n {\sum_{j = 0}^k {( - 1)^j \binom{{n}}{j}2^{ - \left( {n + j - k} \right)} C_{n + j - k + 1} } }
=\begin{cases}
 2^{ - (n - 1)} C_{(n-1)/2},&\text{if $n$ is odd}; \\ 
 0,&\text{if $n$ is even}.  
 \end{cases} 
\end{equation}
In particular
\begin{equation}
\sum_{k = 0}^n \sum_{j = 0}^k ( - 1)^j \binom{{n}}{j} 2^{ - \left( {n + j - k} \right)} C_{n + j - k} = 2^{- (n - 1)} \binom{{n - 1}}{{\left\lfloor {(n - 1)/2} \right\rfloor }}.
\end{equation}
\end{proposition}
\begin{proof}
Use~\eqref{w8utb12} and~\eqref{i1hknxb} in Theorem~\ref{thm.qtnqdr9}.
\end{proof}

\begin{proposition}
If $n$ is a non-negative integer, then
\begin{equation}
\sum_{k = 1}^n {\sum_{j = 0}^{k - 1} {(-1)^jn\binom{{n}}{j}2^{k - j} \frac{{C_{n - k + j + 1} }}{{k - j}}} }  = \begin{cases}
 C_{n + 1},&\text{if $n$ is odd;}  \\ 
 C_{n/2}  - C_{n + 1},&\text{if $n$ is even.}  \\ 
 \end{cases} 
\end{equation}
\end{proposition}
\begin{proof}
Use~\eqref{i1hknxb} in~\eqref{msdi8n1}.
\end{proof}

\begin{proposition}
If $m$ and $n$ are non-negative integers, then
\begin{align}\label{g0638ke}
\sum_{k = 0}^n {\sum_{j = 0}^k {( - 1)^j \binom{{n}}{j}2^{2(k - j)} C_{m + n - k + j + 1} } }  = 2^{2m + 2} \sum_{k = 0}^m {( - 1)^k \binom{{m}}{k}2^{ - 2k} C_{n + k} } 
\end{align}
and
\begin{align}\label{sbl52oj}
&\sum_{k = 1}^n {\sum_{j = 0}^{k - 1} {( - 1)^j n\binom{{n}}{j}\frac{{2^{-2(m + n - k + j)} }}{{k - j}}C_{m + n - k + j + 1} } } \nonumber\\
&\qquad = \sum_{k = 0}^{m - 1} {( - 1)^k \binom{{m}}{k}2^{ - 2(k + n)} C_{k + n + 1} }  + \left( {( - 1)^m  - ( - 1)^n } \right)2^{ - 2(m + n)} C_{m + n + 1}.
\end{align}
In particular,
\begin{equation}
\sum_{k = 0}^n {\sum_{j = 0}^k {( - 1)^j \binom{{n}}{j}2^{2(k - j)} C_{n - k + j + 1} } }  = 4C_n 
\end{equation}
and
\begin{equation}
\sum_{k = 1}^n {\sum_{j = 0}^{k - 1} {( - 1)^j n\binom{{n}}{j}\frac{{2^{-2(n - k + j)} }}{{k - j}}C_{n - k + j + 1} } } 
=\begin{cases}
 2^{1-2n} C_{n+1},&\text{if $n$ is odd}; \\ 
 0,&\text{if $n$ is even.}  
 \end{cases} 
\end{equation}
\end{proposition}
\begin{proof}
Donaghey~\cite{donaghey76} reported the following self-inverse identity:
\begin{equation*}
\sum_{k = 0}^n {( - 1)^k \binom{{n}}{k}\frac{{C_{k + 1} }}{{2^{2k} }}}  = \frac{{C_{n + 1} }}{{2^{2n} }},
\end{equation*}
from which we recognize $s_k=2^{-2k}C_{k+1}=\sigma_k$ whose use in Theorem~\ref{thm.qtnqdr9} produces~\eqref{g0638ke} and in Corollary~\ref{cor.eb6nqko} gives~\eqref{sbl52oj}.
\end{proof}

\begin{lemma}[{\cite[Section 13]{Adegoke0}}]
If $n$ is a non-negative integer, then
\begin{equation}\label{qhaw2rp}
\sum_{k = 0}^n {( - 1)^k \binom{{n}}{k}2^{ - k} C_{k + 1} \left( {H_{k + 2}  - O_{k + 1} } \right)}  =\begin{cases}
 2^{ - n - 1} C_{n/2} H_{(n + 2)/2} ,&\text{if $n$ is even;} \\ 
 0,&\text{if $n$ is odd.} \\ 
 \end{cases}
\end{equation}
\end{lemma}

\begin{proposition}
If $n$ is a non-negative integer, then
\begin{align}
\sum_{k = 0}^n {\sum_{j = 0}^k {( - 1)^j \binom{{n}}{j}2^{ - \left( {j - k} \right)} C_{n + j - k + 1} \left( {H_{n + j - k + 2}  - O_{n + j - k + 1} } \right)} }\nonumber\\  
=\begin{cases}
  C_{(n - 1)/2} H_{(n + 1)/2},&\text{if $n$ is odd;}  \\ 
 0 ,&\text{if $n$ is even.}\\ 
 \end{cases}
\end{align}
\end{proposition}
\begin{proof}
Use~\eqref{qhaw2rp} in~\eqref{utju7bn}.
\end{proof}

\begin{lemma}[{\cite[Section 13]{Adegoke0}}]
If $n$ is a non-negative integer, then
\begin{equation}\label{l6wv97g}
\sum_{k = 0}^n {( - 1)^k \binom{{n}}{k}\left( {2k + 1} \right)\frac{{C_k }}{{2^{2k} }}\frac{{O_{k + 1} }}{{k + 1}}}  = \left( {2n + 1} \right)\frac{{C_n }}{{2^{2n} }}\frac{{O_{n + 1} }}{{n + 1}}.
\end{equation}
\end{lemma}

\begin{proposition}
If $n$ is a positive integer and $m$ is a non-negative integer, then
\begin{align}
\sum_{k = 0}^n {\sum_{j = 0}^k {( - 1)^j \binom{{n}}{j}\left( {2\left( {m + n - k + j} \right) + 1} \right)\frac{{C_{m + n - k + j} }}{{2^{2\left( {m + n - k + j} \right)} }}\frac{{O_{m + n - k + j + 1} }}{{m + n - k + j + 1}}} }\nonumber\\
 = \sum_{k = 0}^m {( - 1)^k \binom{{m}}{k}\left( {2(k + n) - 1} \right)\frac{{C_{k + n - 1} }}{{2^{2(k + n - 1)} }}\frac{{O_{k + n} }}{{k + n}}} .
\end{align}
\end{proposition}
\begin{proof}
Use \eqref{l6wv97g} in Theorem~\ref{thm.qtnqdr9}.
\end{proof}

\subsection{Identities involving Stirling numbers of the second kind}

\begin{lemma}\label{lem.pbovlru}
If $n$, $r$, and $s$ are non-negative integers then
\begin{align}
\sum_{k = 0}^n {( - 1)^k \binom{{n}}{k}k^r }  &= ( - 1)^n n!\braces{{ r}}{n},\label{y7bhb70}\\
\sum_{k = 0}^n {( - 1)^k \binom{{n}}{k}\left( {k + 1} \right)^{r - 1} }  &= ( - 1)^n n!\braces{{ r}}{{n + 1}},\quad r\ge 1\label{prwewl3},
\end{align}
and
\begin{equation}\label{yxdoz00}
\sum_{k = 0}^n {( - 1)^k \binom{{n}}{k}\braces{{ k + r + 1}}{{s + 1}}\binom{{k + r}}{r}^{ - 1} }  = ( - 1)^n \braces{{ n + r}}{s}\binom{{n + r}}{r}^{ - 1},\quad s\ge r .
\end{equation}
\end{lemma}
\begin{proof}
Identity~\eqref{y7bhb70} is the explicit representation of Stirling numbers of the second kind. Identities~\eqref{prwewl3} and~\eqref{yxdoz00} are derived in~\cite[Section 7]{Adegoke0}.
\end{proof}

\begin{proposition}
If $n$, $r$, and $s$ are non-negative integers then
\begin{align}
\sum_{k = 0}^n {\sum_{j = 0}^k {( - 1)^j \binom{{n}}{j}\left( {m + n - k + j} \right)^r } }  &= ( - 1)^{n - 1} \sum_{k = 0}^m {\binom{{m}}{k}\left( {n + k - 1} \right)!\braces{{ r}}{{n + k - 1}}} ,\\
\sum_{k = 0}^n {\sum_{j = 0}^k {( - 1)^j \binom{{n}}{j}\left( {m + n - k + j} \right)^{r - 1} } }  &= ( - 1)^{n - 1} \sum_{k = 0}^m {\binom{{m}}{k}\left( {n + k - 1} \right)!\braces{{ r}}{{n + k}}} ,\quad r\ge 1,
\end{align}
and
\begin{align}\label{mgi6qsx}
&\sum_{k = 0}^n {\sum_{j = 0}^k {( - 1)^j \binom{{n}}{j}\braces{{ m + n - k + j + r + 1}}{{s + 1}}\binom{{m + n - k + j + r}}{r}^{ - 1} } }\nonumber\\
&\qquad  = (-1)^{n-1}\sum_{k = 0}^m {( - 1)^k \binom{{m}}{k}\braces{{ n + k + r - 1}}{s}\binom{{n + k + r - 1}}{r}^{ - 1} }\quad s\ge r .
\end{align}
\end{proposition}
\begin{proof}
These results follow from the use of the identities in Lemma~\ref{lem.pbovlru} in Theorem~\ref{thm.qtnqdr9}.
\end{proof}


\begin{proposition}
If $n$ is a positive integer and $r$ and $s$ are non-negative integers then
\begin{align}
\sum_{k = 1}^n {\sum_{j = 0}^{k - 1} {( - 1)^j \binom{{n}}{j}\frac{{\left( {n - k + j} \right)^r }}{{k - j}}} }  &= ( - 1)^n (n - 1)!\braces{{ r}}{n} - ( - 1)^n n^{r - 1}\label{ih1affl} ,\\
\sum_{k = 1}^n {\sum_{j = 0}^{k - 1} {( - 1)^j \binom{{n}}{j}\frac{{\left( {n - k + j + 1} \right)^{r - 1} }}{{k - j}}} }  &= ( - 1)^n (n - 1)!\braces{{ r}}{{n + 1}} - ( - 1)^n \frac{{\left( {n + 1} \right)^{r - 1} }}{n},\quad r\ge 1\label{bjpvnnv},
\end{align}
and
\begin{align}
&\sum_{k = 1}^n {\sum_{j = 0}^{k - 1} {( - 1)^j \binom{{n}}{j}\braces{{ n - k + j + r + 1}}{{s + 1}}\binom{{n - k + j + r}}{r}^{ - 1} \frac{1}{{k - j}}} }\nonumber\\
&\qquad  = ( - 1)^{n - 1} \left( {s + 1} \right)\braces{{ n + r}}{{s + 1}}\binom{{n + r}}{r}^{ - 1},\quad s\ge r .
\end{align}
\end{proposition}
\begin{proof}
Use the identities stated in Lemma~\ref{lem.pbovlru} in~\eqref{msdi8n1}.
\end{proof}

In particular, setting $r=n-1$ in~\eqref{ih1affl} gives
\begin{equation}
\sum_{k = 1}^n {\sum_{j = 0}^{k - 1} {( - 1)^j \binom{{n}}{j}\frac{{\left( {n - k + j} \right)^{n - 1} }}{{k - j}}} }  = ( - 1)^{n + 1} n^{n - 2} ,\quad n\in\mathbb Z^+,
\end{equation}
while setting $r=n$ in~\eqref{bjpvnnv} yields
\begin{equation}
\sum_{k = 1}^n {\sum_{j = 0}^{k - 1} {( - 1)^j \binom{{n}}{j}\frac{{\left( {n - k + j + 1} \right)^{n - 1} }}{{k - j}}} }  = ( - 1)^{n + 1} \frac{{\left( {n + 1} \right)^{n - 1} }}{n},\quad n\in\mathbb Z^+.
\end{equation}

\begin{lemma}
If $n$ is a non-negative integer, $m$ is a non-negative integer and $r$ is a complex number that is not a negative integer, then~\cite[Equation (10.45)]{Boya4}:
\begin{equation}\label{wk165o9}
\sum_{k = 0}^n {( - 1)^k \binom{{n}}{k}\binom{{k + r}}{k}^{ - 1} k^m }  = \binom{{ - r - 1}}{n}^{ - 1} \sum_{k = 0}^m {\binom{{ - r}}{{n - k}}\braces{{ m}}{k}k!} 
\end{equation}
\end{lemma}

\begin{proposition}
If $n$ is a positive integer, $m$ is a non-negative integer and $r$ is a complex number that is not a negative integer, then
\begin{equation}
\sum_{k = 0}^n {\sum_{j = 0}^k {( - 1)^j \binom{{n}}{j}\binom{{n + j - k + r}}{r}^{ - 1} \left( {n + j - k} \right)^m } }  = \binom{{ - r - 1}}{{n - 1}}^{ - 1} \sum_{k = 0}^m {\binom{{ - r}}{{n - k - 1}}\braces{{ m}}{k}k!} 
\end{equation}
\end{proposition}
\begin{proof}
Use~\eqref{wk165o9} in Theorem~\ref{thm.qtnqdr9}.
\end{proof}

\subsection{Identities involving $m$-step numbers}

\begin{lemma}[See~\cite{Adegoke18c}]\label{lem.fwavfkb}
Let $W_k^{(m)}$ be a Fibonacci or Lucas $m$-step number. Then
\begin{equation}
\sum_{k = 0}^n {( - 1)^k \binom{{n}}{k}2^k W_{mk}^{(m)} }  = ( - 1)^n W_{(m + 1)n}^{(m)} 
\end{equation}
and
\begin{equation}
\sum_{k = 0}^n {( - 1)^k \binom{{n}}{k}2^{ - k} W_k^{(m)} }  = 2^{ - n} W_{ - mn}^{(m)} .
\end{equation}
\end{lemma}
The identities stated in Propositions~\ref{prop.ft6yauz} and~\ref{prop.up76gra} are immediate consequences of Theorem~\ref{thm.qtnqdr9} and Lemma~\ref{lem.fwavfkb}.
\begin{proposition}\label{prop.ft6yauz}
If $n$ is a non-negative integer, then
\begin{equation}
\sum_{k = 0}^n {\sum_{j = 0}^k {( - 1)^j \binom{{n}}{j}2^{n + j - k} W_{m(n + j - k)}^{(m)} } }  = ( - 1)^{n - 1} W_{(m + 1)(n - 1)}^{(m)} .
\end{equation}
In particular,
\begin{equation}
\sum_{k = 0}^n {\sum_{j = 0}^k {( - 1)^j \binom{{n}}{j}2^{n + j - k} G_{2(n + j - k)} } }  = ( - 1)^{n - 1} G_{3(n - 1)} 
\end{equation}
and
\begin{equation}
\sum_{k = 0}^n {\sum_{j = 0}^k {( - 1)^j \binom{{n}}{j}2^{n + j - k} T_{3(n + j - k)} } }  = ( - 1)^{n - 1} T_{4(n - 1)} .
\end{equation}
\end{proposition}

\begin{proposition}\label{prop.up76gra}
If $n$ is a non-negative integer, then
\begin{equation}
\sum_{k = 0}^n {\sum_{j = 0}^k {( - 1)^j \binom{{n}}{j}2^{k - j} W_{n + j - k}^{(m)} } }  = 2W_{ - m(n - 1)}^{(m)} .
\end{equation}
\end{proposition}
In particular,
\begin{equation}
\sum_{k = 0}^n {\sum_{j = 0}^k {( - 1)^j \binom{{n}}{j}2^{k - j} G_{n + j - k} } }  = 2G_{ - 2(n - 1)} 
\end{equation}
and
\begin{equation}
\sum_{k = 0}^n {\sum_{j = 0}^k {( - 1)^j \binom{{n}}{j}2^{k - j} T_{n + j - k} } }  = 2T_{ - 3(n - 1)} .
\end{equation}

\section{Identities involving classical polynomials}

This section contains applications of our results from Section 3 to classical polynomials. 
Due to the high number of new identities that we could state we restrict the exposition to some basis applications.

\subsection{Identities involving Fibonacci polynomials, Lucas polynomials, Chebyshev polynomials and more}

First we recall some facts about Horadam sequences. We do this for reasons that will become obvious below. 
The Horadam sequence $w_n = w_n(a,b;p,q)$ is defined, for all integers, by the recurrence relation~\cite{Horadam65}
\begin{equation*}\label{Def-Horadam}
w_0 = a,\,\, w_1 = b, \quad w_n = p w_{n - 1} - q w_{n - 2},\quad n\ge 2,
\end{equation*}
with
\begin{equation*}
w_{-n} = \frac{1}{q}(pw_{-n+1} - w_{-n+2})\,,
\end{equation*}
where $a$, $b$, $p$ and $q$ are arbitrary complex numbers with $p \neq 0$, $q \neq 0$, and $p^2 - 4q > 0$. 
The $n$-th term of a Horadam sequence is given by 
\begin{equation*}\label{Horadam_Binet}
w_n = w_n(a,b;p,q) = \frac{A \mu^n(p,q) - B \nu^n(p,q)}{\mu(p,q) - \nu(p,q)},
\end{equation*}
where
\[ A = w_1 - w_0 \nu(p,q), \quad B = w_1 - w_0 \mu(p,q), \]
and $\mu(p,q)$ and $\nu(p,q)$ are given by
\[ \mu = \mu(p,q) = \frac{p + \delta}{2}, \quad \nu = \nu(p,q) = \frac{p - \delta}{2}, \]
where $\delta=\sqrt{p^2 - 4q}$, so that $\mu(p,q)\,\nu(p,q)=q$. 
The sequence $w_n$ not only generalizes many important number sequences but also contains many polynomial sequences as special members: 
\begin{itemize}
\item $w_n(0,1;x,-1)=F_n(x)$ are the Fibonacci polynomials
\item $w_n(2,x;x,-1)=L_n(x)$ are the Lucas polynomials
\item $w_n(0,1;2x,-1)=P_n(x)$ are the Pell polynomials
\item $w_n(2,2x;2x,-1)=Q_n(x)$ are the Pell-Lucas polynomials
\item $w_n(1,x;2x,1)=T_n(x)$ are the Chebyshev polynomials of the first kind
\item $w_n(1,2x;2x,1)=U_n(x)$ are the Chebyshev polynomials of the second kind, and so on. 
\end{itemize}
We also mention the less known members
\begin{equation}\label{powers}
w_n(0,1;x+1,x)=\frac{x^n-1}{x-1} \qquad\text{and}\qquad w_n(2,x+1;x+1,x)=x^n-1,
\end{equation}
the first one being commonly called the base-$x$ repunits.

\begin{proposition}
If $n$ is a positive integer, then
\begin{equation}
\sum_{k=0}^n \sum_{j=0}^k \binom{n}{j} (-1)^j \left (\frac{p}{q}\right )^{n-k+j} w_{n-k+j} = (-1)^{n-1} q^{-(n-1)} w_{2(n-1)}.
\end{equation}
\end{proposition}
\begin{proof}
Apply Theorem \ref{thm.qtnqdr9} to the binomial transform pair \cite{Frontczak}
\begin{equation}\label{Hor_bin1}
\{s_n,\sigma_n\} = \{ (p/q)^{n} w_{n}, (-1)^n q^{-n} w_{2n} \}.
\end{equation}
\end{proof}

We skip the explicit presentation of the double sums associated with the classical polynomials. 
It is worth noticing, however, that if we work with \eqref{powers} we get
\begin{equation}
\sum_{k=0}^n \sum_{j=0}^k \binom{n}{j} (-1)^j (1+x)^{n-k+j} \left (1+\frac{1}{x^{n-k+j}}\right ) 
= (-1)^{n-1} \left (x^{n-1}+\frac{1}{x^{n-1}}\right )
\end{equation}
and
\begin{equation}
\sum_{k=0}^n \sum_{j=0}^k \binom{n}{j} (-1)^j (1+x)^{n-k+j} \left (1-\frac{1}{x^{n-k+j}}\right ) 
= (-1)^{n-1} \left (x^{n-1}-\frac{1}{x^{n-1}}\right ).
\end{equation}
Using the transformation $x\mapsto e^x$ and simplifying we derive at two new the identities involving hyperbolic functions.
\begin{corollary}
If $n$ is a positive integer, then
\begin{equation}
\sum_{k=0}^n \sum_{j=0}^k \binom{n}{j} (-1)^j 2^{n-k+j} \cosh^{n-k+j}\left (\frac{x}{2}\right ) \cosh\left ((n-k+j)\frac{x}{2}\right )
= (-1)^{n-1} \cosh ((n-1)x)
\end{equation}
and
\begin{equation}
\sum_{k=0}^n \sum_{j=0}^k \binom{n}{j} (-1)^j 2^{n-k+j} \cosh^{n-k+j}\left (\frac{x}{2}\right ) \sinh\left ((n-k+j)\frac{x}{2}\right )
= (-1)^{n-1} \sinh ((n-1)x).
\end{equation}
\end{corollary}

\begin{proposition}
If $n$ is a positive integer, then
\begin{equation}
\sum_{k=1}^n \sum_{j=0}^{k-1} \binom{n}{j} n (-1)^j \left (\frac{p}{q}\right )^{n-k+j} \frac{w_{n-k+j}}{k-j} 
= (-1)^{n} q^{-n}( w_{2n} - p^n w_n ).
\end{equation}
\end{proposition}
\begin{proof}
Apply Corollary \ref{cor.cp9d6at} to the binomial transform pair \eqref{Hor_bin1}.
\end{proof}

We mention explicitly the special cases associated with \eqref{powers}.
\begin{corollary}
If $n$ is a positive integer, then
\begin{align}
&\sum_{k=1}^n \sum_{j=0}^{k-1} \binom{n}{j} n (-1)^j 2^{n-k+j} \cosh^{n-k+j}\left (\frac{x}{2}\right ) \frac{\cosh\left ((n-k+j)\frac{x}{2}\right )}{k-j} \nonumber \\
&\qquad\qquad\qquad\qquad = (-1)^{n} \left (\cosh (nx) - 2^n \cosh^n \left (\frac{x}{2}\right ) \cosh\left (\frac{nx}{2}\right )\right )
\end{align}
and
\begin{align}
&\sum_{k=1}^n \sum_{j=0}^{k-1} \binom{n}{j} n (-1)^j 2^{n-k+j} \cosh^{n-k+j}\left (\frac{x}{2}\right ) \frac{\sinh\left ((n-k+j)\frac{x}{2}\right )}{k-j} \nonumber \\
&\qquad\qquad\qquad\qquad = (-1)^{n} \left (\sinh (nx) - 2^n \cosh^n \left (\frac{x}{2}\right ) \sinh\left (\frac{nx}{2}\right )\right ).
\end{align}
\end{corollary}

\begin{proposition}
If $n$ is a positive integer, then
\begin{equation}
\sum_{k=0}^n \sum_{j=0}^{k} \binom{n}{j} (-1)^j \left (\frac{p}{q}\right )^{k-j} w_{k-j} = q^{-n} w_{2n} + q^{-(n-1)} w_{2(n-1)}.
\end{equation}
\end{proposition}
\begin{proof}
Apply Corollary \ref{cor_Ber_fin} to the binomial transform pair \eqref{Hor_bin1}.
\end{proof}

The special cases associated with \eqref{powers} are stated next.

\begin{corollary}
If $n$ is a positive integer, then
\begin{equation}
\sum_{k=0}^n \sum_{j=0}^{k} \binom{n}{j} (-1)^j 2^{k-j} \cosh^{k-j}\left (\frac{x}{2}\right ) \cosh\left ((k-j)\frac{x}{2}\right )
= \cosh (nx) + \cosh ((n-1)x)
\end{equation}
and
\begin{equation}
\sum_{k=0}^n \sum_{j=0}^{k} \binom{n}{j} (-1)^j 2^{k-j} \cosh^{k-j}\left (\frac{x}{2}\right ) \sinh\left ((k-j)\frac{x}{2}\right )
= \sinh (nx) + \sinh ((n-1)x).
\end{equation}
\end{corollary}

\begin{proposition}
If $n$ is a positive integer, then
\begin{equation}
\sum_{k=0}^n \sum_{j=0}^{k} \binom{n+1}{j+1} (-1)^{j+1} \left (\frac{p}{q}\right )^{j+1} w_{j+1} 
= (n+1) (-1)^n q^{-n} w_{2n} + (-1)^{n+1} q^{-(n+1)} w_{2(n+1)} - (n+2)a.
\end{equation}
\end{proposition}
\begin{proof}
Apply Corollary \ref{int_cor} to the binomial transform pair \eqref{Hor_bin1}.
\end{proof}

\begin{corollary}
If $n$ is a positive integer, then
\begin{align}
&\sum_{k=0}^n \sum_{j=0}^{k} \binom{n+1}{j+1} (-1)^{j+1} 2^{j+1} \cosh^{j+1}\left (\frac{x}{2}\right ) \cosh\left ((j+1)\frac{x}{2}\right ) \nonumber \\
&\qquad\qquad = (-1)^n \left ((n+1)\cosh (nx) - \cosh ((n+1)x)\right ) -(n+2)
\end{align}
and
\begin{align}
&\sum_{k=0}^n \sum_{j=0}^{k} \binom{n+1}{j+1} (-1)^{j+1} 2^{j+1} \cosh^{j+1}\left (\frac{x}{2}\right ) \sinh\left ((j+1)\frac{x}{2}\right ) \nonumber \\
&\qquad\qquad = (-1)^n \left ((n+1)\sinh (nx) - \sinh ((n+1)x)\right ).
\end{align}
\end{corollary}

\begin{remark}
Many more double sum identities involving classical polynomials can be derived by making use of the binomial transform pairs \cite{Frontczak}
\begin{align*}
\left\{s_n,\sigma_n \right\} &= \left\{\left (\frac{p^2-q}{pq}\right )^n w_{n}, (-pq)^{-n} w_{3n} \right\}, \\
\left\{s_n,\sigma_n \right\} &= \left\{\left (\frac{p^2-q}{q^2}\right )^n w_{2n}, (-1)^n \left(\frac{p}{q^2}\right)^{n} w_{3n}, \right\}, \\
\left\{s_n,\sigma_n \right\} &= \left\{\left (\frac{p(p^2-2q)}{q(p^2-q)}\right )^n w_{n}, (-1)^n (q(p^2-q))^{-n} w_{4n},\right\}, \\
\left\{s_n,\sigma_n \right\} &= \left\{\left (\frac{p^2-2q}{q^2}\right )^n w_{2n}, (-1)^n q^{-2n} w_{4n}, \right\}. 
\end{align*}
\end{remark}

\subsection{Identities involving Bernoulli polynomials}

\begin{proposition}
If $n$ is a positive integer, $m$ is a non-negative integer and $x$ and $y$ are non-zero complex numbers, then
\begin{equation}\label{t6lc9ok}
\sum_{k = 0}^n {( - 1)^k \sum_{j = 0}^k {\binom{{n}}{j}\frac{{B_{m + n - k + j} (y)}}{{x^{m - k + j} }}} }  = ( - 1)^{m + n} \sum_{k = 0}^m {( - 1)^k \binom{{m}}{k}\frac{{B_{n + k - 1} (x + y)}}{{x^{k - 1} }}} .
\end{equation}
In particular,
\begin{equation}
\sum_{k = 0}^n {( - 1)^k \sum_{j = 0}^k {\binom{{n}}{j}x^{k - j} B_{n - k + j}(y) } }  = ( - 1)^n xB_{n - 1} (x + y).
\end{equation}
\end{proposition}
\begin{proof}
From the complement theorem of the Bernoulli polynomials~\eqref{de8ucu6}, we can identify the binomial transform pair:
\begin{equation}\label{kc11oa4}
s_n  = ( - 1)^n x^{ - n} B_n (y) \qquad\text{and}\qquad \sigma_n = x^{- n} B_n (x + y),
\end{equation}
which when used in Theorem~\ref{thm.qtnqdr9} produces~\eqref{t6lc9ok}.
\end{proof}

\begin{proposition}
If $m$ and $n$ are non-negative integers and $x$ and $y$ are complex numbers, then
\begin{align}
&\sum_{k = 1}^n {( - 1)^k \sum_{j = 0}^{k - 1} {n\binom nj\frac{{x^{k - j} }}{{k - j}}B_{m + n - k + j} (y)} }\nonumber\\
&\qquad  = ( - 1)^{m + n} \sum_{k = 0}^m {( - 1)^k \binom{{m}}{k}x^{m - k} B_{k + n} (x + y)}  - ( - 1)^n B_{m + n} (y).
\end{align}
In particular,
\begin{equation}
\sum_{k = 1}^n {( - 1)^k \sum_{j = 0}^{k - 1} {n\binom nj\frac{{x^{k - j} }}{{k - j}}B_{n - k + j} (y)} }=(-1)^n\left(B_n(x+y)-B_n(y)\right).
\end{equation}
\end{proposition}
\begin{proof}
Use~\eqref{kc11oa4} in Corollary~\ref{cor.cp9d6at}.
\end{proof}

\begin{proposition}
If $n$ is a positive integer, $m$ is a non-negative integer and $x$ and $y$ are non-zero complex numbers, then
\begin{align}
&\sum_{k = 0}^n {( - 1)^k \sum_{j = 0}^k {\binom{{n}}{j}x^{n - m - k + j} B_{m + k - j} (y)} }\nonumber\\
&\qquad  = ( - 1)^{m - n} \sum_{k = 0}^m {( - 1)^k \binom{{m}}{k}x^{ - k} \left( {B_{k + n} (x + y) - xB_{k + n - 1} (x + y)} \right)} .
\end{align}
In particular,
\begin{equation}
\sum_{k = 0}^n {( - 1)^k \sum_{j = 0}^k {\binom{{n}}{j}x^{n - k + j} B_{k - j} (y)} }=(-1)^n\left(B_n(x+y)-xB_{n-1}(x+y)\right).
\end{equation}
\end{proposition}
\begin{proof}
Use~\eqref{kc11oa4} in Corollary~\ref{cor_Ber_fin}.
\end{proof}

\begin{lemma}[{\cite[Section 13]{Adegoke0}}]
If $n$ and $r$ are non-negative integers, then
\begin{equation}\label{etio67p} 
\sum_{k = 0}^n {( - 1)^k \binom{{n}}{k}B_r \left( {k + 1} \right)}  = ( - 1)^n (n - 1)!r\braces{{ r}}{n},\quad n\ne0,
\end{equation}
and
\begin{equation}\label{cmg5y2b}
\sum_{k = 0}^n {( - 1)^k \binom{{n}}{k}\frac{{B_r \left( {k + 1} \right)}}{{k + 1}}}  = \frac{{B_r }}{{n + 1}} + \frac{{( - 1)^n n!}}{{n + 1}}r\braces{{ r - 1}}{n}.
\end{equation}
\end{lemma}

\begin{proposition}
If $n$ is a positive integer greater than or equal to $2$ and $m$ and $r$ are non-negative integers, then
\begin{align}
&\sum_{k = 0}^n {\sum_{j = 0}^k {( - 1)^j \binom{{n}}{j}B_r \left( {m + n - k + j + 1} \right)} }\nonumber\\
&\qquad  = ( - 1)^{n - 1} r\sum_{k = 0}^m {( - 1)^k \binom{{m}}{k}\left( {n + k - 2} \right)!\braces{{ r}}{{n + k - 1}}} 
\end{align}
and
\begin{align}\label{u0jw006}
&\sum_{k = 0}^n {\sum_{j = 0}^k {( - 1)^j \binom{{n}}{j}\frac{{B_r \left( {m + n - k + j + 1} \right)}}{{m + n - k + j + 1}}} }\nonumber\\
&\qquad= \frac{{B_r }}{{m + n}}\binom{{m + n - 1}}{m}^{ - 1}+ ( - 1)^{n - 1} r\sum_{k = 0}^m {\frac{{\left( {n + k - 1} \right)!}}{{n + k}}\binom{{m}}{k}\braces{{ r - 1}}{{n + k - 1}}} 
\end{align}
\end{proposition}
\begin{proof}
From~\eqref{etio67p} and~\eqref{cmg5y2b}, we identify the binomial transform pairs
\begin{equation}\label{qtczn06}
s_n = B_r (n + 1) \qquad\text{and}\qquad \sigma_n = ( - 1)^n (n - 1)!r\braces{{ r}}{n},
\end{equation}
and
\begin{equation}
s_n = \frac{{B_r (n + 1)}}{{n + 1}} \qquad\text{and}\qquad \sigma_n = \frac{{( - 1)^n n!}}{{n + 1}}r\braces{{ r - 1}}{n},
\end{equation}
and use these in Theorem~\ref{thm.qtnqdr9}. In simplifying~\eqref{u0jw006}, we used
\begin{equation}
\sum_{k = 0}^m {( - 1)^k \binom{{m}}{k}\frac{1}{{n + k}}}  = \frac{1}{{m + n}}\binom{{m + n - 1}}{m}^{ - 1}.
\end{equation}
\end{proof}

\begin{proposition}
If $n$ is a positive integer and $m$ and $r$ are non-negative integers, then
\begin{align}
&\sum_{k = 1}^n {\sum_{j = 0}^{k - 1} {( - 1)^j n\binom{{n}}{j}\frac{{B_r \left( {m + n - k + j + 1} \right)}}{{k - j}}} }\nonumber\\
&\qquad  = ( - 1)^n r\sum_{k = 0}^m {\binom{{m}}{k}\braces{{ r}}{{k + n}}(k+n-1)!}  - ( - 1)^n B_r \left( {m + n + 1} \right).
\end{align}
In particular,
\begin{equation}
\sum_{k = 1}^n {\sum_{j = 0}^{k - 1} {( - 1)^j n\binom{{n}}{j}\frac{{B_r \left( {n - k + j + 1} \right)}}{{k - j}}} }=(-1)^nr\braces rn(n-1)!-(-1)^nB_r(n+1).
\end{equation}
\end{proposition}
\begin{proof}
Use~\eqref{qtczn06} in Corollary~\ref{cor.cp9d6at}.
\end{proof}

\section*{Funding}
This research did not receive any specific grant from funding agencies in the public, commercial, or not-for-profit sectors.

\section*{Competing interests}
The authors have no competing interests to declare that are relevant to the content of this article.


\begin{thebibliography}{99}


\bibitem{Adegoke18c} 
K. Adegoke, Linear properties of generalized $n-$step Fibonacci numbers, Preprint, (2018). arXiv:1808.02878[math.NT] 

\bibitem{Adegoke21}
K. Adegoke,  A. Olatinwo and S. Gosh, Cubic binomial Fibonacci sums, \emph{Electronic Journal of Mathematics} {\bf 2} (2021), 44--51.

\bibitem{Adegoke0}
K. Adegoke, Binomial transforms and the binomial convolution of sequences, Preprint, (2025). arXiv:2507.04179[math.NT]

\bibitem{Adegoke1}
K. Adegoke and R. Frontczak, Some notes on an identity of Frisch, Open J. Math. Sci. 8 (2024), 216--226.

\bibitem{Adegoke2}
K. Adegoke, R. Frontczak and K. Gryszka, Finite sums associated with some polynomial identities, 
Integral Transforms Spec. Funct. (2025), DOI:10.1080/10652469.2025.2529410

\bibitem{Adegoke3}
K. Adegoke, R. Frontczak and K. Gryszka, Double sums involving binomial coefficients and special numbers, Preprint, (2025), under review. 

\bibitem{Boya1} 
K. N. Boyadzhiev, Harmonic number identities via Euler's transform, J. Integer Seq. 12 (2009), Article 09.6.1.

\bibitem{Boya2}
K. N. Boyadzhiev, Binomial transform and the backward difference, Adv. Appl. Discrete Math. 18 (2014), 43--63.

\bibitem{Boya3} 
K. N. Boyadzhiev, Binomial transform of products, Ars Combinatoria 126 (2016), 415--434.

\bibitem{Boya4}
K. N. Boyadzhiev, \emph{Notes on the Binomial Transform}, World Scientific, (2018). 

\bibitem{Camargo21} 
A. P. Carmago and G. C. O. Teruya, A few remarks on the values of the Bernoulli polynomials at rational arguments and some relations with $\zeta(2k+1)$, \emph{Notes on Number Theory and Discrete Mathematics} {\bf 27}:4 (2021), 180--186.

\bibitem{chen07} 
K. W. Chen, Identities from the binomial transform, J. Number Theory 124 (2007), 142--150.

\bibitem{donaghey76} R. Donaghey, Binomial self-inverse sequences and tangent coefficients, \emph{Journal of Combinatorial Theory (A)} {\bf 21} (1976), 155--163.

\bibitem{Frontczak}
R. Frontczak, A short remark on Horadam identities with binomial coefficients, Ann. Math. Inform. 54 (2021), 5--13.

\bibitem{Gould} 
H. W. Gould, \emph{Combinatorial Identities}, West Virginia University, Morgantown, 1972. 

\bibitem{Gould2} 
H. W. Gould, Series transformations for finding recurrences for sequences, Fibonacci Quart. 28 (1990), 166--171.

\bibitem{gould14} 
H. W. Gould and J. Quaintance, Bernoulli numbers and a new binomial transform identity, J. Integer Seq. 17 (2014), Article 14.2.2.

\bibitem{Graham}
R. L. Graham, D.E. Knuth and O. Patashnik, \emph{Concrete mathematics: A foundation for computer science}, Second Edition, Addison-Wesley, Reading, 2022.

\bibitem{Haukkanen}
P. Haukkanen, Formal power series for binomial sums of sequences of numbers, Fibonacci Quart. 31 (1993), 28--31.

\bibitem{Horadam65}
A. F. Horadam, Basic properties of a certain generalized sequence of numbers, Fibonacci Quart. 3 (1965), 161--176.

\bibitem{Koshy}
T. Koshy, \emph{Fibonacci and Lucas Numbers with Applications}, Wiley-Interscience, 2001.

\bibitem{mathworld} 
R. Stanley and E. W. Weisstein, Catalan number, \emph{MathWorld - A Wolfram Resource}, \url{https://mathworld.wolfram.com/CatalanNumber.html}.

\bibitem{Mu}
Y.- P. Mu, Symmetric recurrence relations and binomial transforms, J. Number Theory 133 (9) (2013), 3127--3137.


\bibitem{OEIS}
N. J. A. Sloane (editor), The On-Line Encyclopedia of Integer Sequences, 2022. Available at \url{https://oeis.org}.

\bibitem{Prodinger}
H. Prodinger, Some information about the binomial transform, Fibonacci Quart. 32 (1994), 412--415.

\bibitem{Riordan} 
J. Riordan, \emph{Combinatorial Identities}, John Wiley \& Sons, Inc., New York, 1968. 

\bibitem{Srivastava}
H. M. Srivastava and J. Choi, \emph{Series Associated with the Zeta and Related Functions}, Springer Science+Media, B.V., 2001.

\bibitem{Stenlund}
D. Stenlund and J. G. Wan, Some double sums involving ratios of binomial coefficients arising from urn models,
J. Integer Seq. 22 (2019), Article 19.1.8.

\bibitem{sulei23} 
E. Suleiman and B. Sury, A potpourri of plenty identities involving Catalan, Fibonacci and trigonometric numbers, Parabola 59 (2023).

\bibitem{Sun1}
Z.-H. Sun, Invariant sequences under binomial transformation, Fibonacci Quart. 39 (2001), 324--333.

\bibitem{Sun2}
Z.-W. Sun, Combinatorial identities in dual sequences, European J. Comb. 24 (2003), 709--718.

\bibitem{Sun3}
Z.-H. Sun, Some further properties of even and odd sequences, Int. J. Number Theory 13 (2017), 1419--1442.

\bibitem{Vajda}
S. Vajda, \emph{Fibonacci and Lucas Numbers, and the Golden Section: Theory and Applications}, Dover Press, 2008.

\bibitem{Wang}
Y. Wang, Self-inverse sequences related to a binomial inverse pair, Fibonacci Quart. 43 (2005), 46--52.




\end{thebibliography}
\end{document}